\documentclass [12pt,a4paper,preprint]{elsarticle}

\usepackage{mathtools}
\usepackage{tikz,tikz-3dplot}
\usepackage{stix}
\usepackage{verbatim}
\usepackage[top=1in, bottom=1.25in, left=1.25in, right=1.25in]{geometry}
\usepackage{enumitem}
\usepackage{amsthm}
\usepackage{graphicx}
\usepackage{mathtools}
\usepackage{bbm}
\usepackage{amsmath}
\newcommand{\measurerestr}{%
  \,\raisebox{-.127ex}{\reflectbox{\rotatebox[origin=br]{-90}{$\lnot$}}}\,%
}

\newcommand{\R}{\mathbb{R}}
\newcommand{\Le}{\mathscr{L}}

\newcommand{\T}{\mathbb{T}}
\newcommand*\circled[1]{\tikz[baseline=(char.base)]{
            \node[shape=circle,draw,inner sep=2pt] (char) {#1};}}

\AtBeginDocument{\let\div\undefined\DeclareMathOperator{\div}{div}}
\newcommand{\E}{\mathbb{E}}
\newtheorem{definition}{Definition}[section]
\newtheorem*{theorem*}{Theorem}
\newtheorem{theorem}{Theorem}[section]
\newtheorem{corollary}{Corollary}[section]
\newtheorem{lemma}{Lemma}[section]

\newtheorem{remark}{Remark}
\newtheorem{proposition}{Proposition}[section]
\numberwithin{equation}{section}
\newtheorem{question}{Question}
\providecommand{\abstract}{}
\DeclareMathOperator{\W}{W}
\AtBeginDocument{\let\d\undefined\DeclareMathOperator{\d}{d}}
\AtBeginDocument{\let\P\undefined\DeclareMathOperator{\P}{P}}
\makeatletter
\def\ps@pprintTitle{%
 \let\@oddhead\@empty
 \let\@evenhead\@empty
 \def\@oddfoot{\reset@font\hfil\thepage\hfil}%
 \let\@evenfoot\@oddfoot}
\makeatother

\date{}
\begin{document}
\begin{frontmatter}
\title{AN EXPLICIT EULER METHOD FOR SOBOLEV VECTOR FIELDS WITH APPLICATIONS TO THE CONTINUITY EQUATION ON NON CARTESIAN GRIDS}

\author[Scuola Normale Superiore]{Tommaso Cortopassi}

\ead{tommaso.cortopassi@sns.it}

\affiliation[Scuola Normale Superiore]{organization={Scuola Normale Superiore},
            addressline={Piazza Dei Cavalieri 7},
            city={Pisa},
            postcode={56126}, 
            country={Italy}}

\begin{abstract}
We prove a novel stability estimate in $L^\infty _t (L^p_x)$ between the regular Lagrangian flow of a Sobolev vector field and a piecewise affine approximation of such flow. This approximation of the flow is obtained by a (sort of) explicit Euler method, and it is the crucial tool to prove approximation results for the solution of the continuity equation by using the representation of the solution as the push-forward via the regular Lagrangian flow of the initial datum. We approximate the solution in two ways, using different approximations for both the flow and the initial datum. In the first case we give an estimate, which however holds only in probability, of the Wasserstein distance between the solution of the continuity equation and a discrete approximation of such solution. The approximate solution is defined as the push-forward of weighted Dirac deltas (whose centers are chosen in a probabilistic way). In the second case we give a deterministic estimate of the Wasserstein distance using a slightly different approximation of the regular Lagrangian flow and requiring more regularity on the velocity field $u$ than in the previous case. An advantage of both approximations is that they provide an algorithm which is easily parallelisable and does not rely on any particular structure of the mesh with which we discretise (only in space) the domain, moreover the method does not need a CFL condition to be satisfied. We also compare our estimates to similar ones previously obtained in \cite{schlichting2017convergence}, and we show how under certain hypotheses our method provides better convergence rates. 
\end{abstract}
\end{frontmatter}

\section{Introduction}
One of the most important results in classical analysis is the Cauchy-Lipschitz theorem:

\begin{theorem}{\textbf{(Cauchy-Lipschitz) \footnote{There exist many different versions of this theorem. The boundedness of $u$ is not usually required, but we assume it since in the following the velocity $u$ that we consider will always be bounded.}}}\label{Cauchy-Lipschitz theorem}

    Let $u : [0,T] \times \mathbb{R}^d  \to \mathbb{R}^d$ be bounded, continuous in $t$ and Lipschitz in $x$, uniformly in time. Then, for every $x_0 \in \mathbb{R}^d $ and for every $T>0$, the solution of the Cauchy problem

    \begin{equation}\label{Cauchy ODE}
    \begin{cases}
        \frac{d}{dt} y(t) = u(t, y(t))\\
        y(0)= x_0 
    \end{cases}
    \end{equation}

    exists for all $t \in [0, T]$ and it is unique.
\end{theorem}

The Cauchy problem \eqref{Cauchy ODE} is strictly connected with the following initial value problem for the continuity equation:

\begin{equation}\label{continuity equation in Rd}
    \begin{cases}
        \partial_t \rho + \div_x (u \rho)=0 &\text{ in } (0,T) \times \R^d \\
        \rho(0, \cdot )= \rho _0 &\text{ in }  \R^d .
    \end{cases}
\end{equation}

The link between \eqref{Cauchy ODE} and \eqref{continuity equation in Rd} in the classical setting (namely when the velocity field $u$ satisfies the assumptions of Theorem \ref{Cauchy-Lipschitz theorem}) is provided by the method of characteristics, a well-known tool for first order linear PDEs  which shows that the solution $\rho$ of \eqref{continuity equation in Rd} can be represented as

$$ \rho(t, \cdot)= \Phi(t, \cdot) _\#  \rho _0 ,$$

where $\Phi$ is the flow of $u$, i.e. $\Phi(\cdot,x)$ satisfies for every $x$: 

$$ \begin{cases}
    \partial_t \Phi(t,x)= u(t,\Phi(t,x)) &\text{if } t>0\\
    \Phi(0,x)=x  &\text{if } t=0.
\end{cases}$$

It is well-known that in \eqref{Cauchy ODE} the assumption that $u$ is continuous in time can be removed asking for example only summability, but it is also important for several applications to significantly weaken also the Lipschitz regularity condition in the space variable. Indeed \eqref{continuity equation in Rd} is often part of many equations coming from physics (a very important example is fluid dynamics, where the continuity equation is part of the inhomogeneous Navier-Stokes equations) which encodes the ``conservation of mass'', i.e. the quantity $\|\rho(t, \cdot)\|_{L^1 (\R^d)}$ is constant in time, and in this setting the velocity does not have Lipschitz regularity in general. The first breakthrough in this direction \footnote{Some progress had in fact already been made, proving well-posedness for \eqref{Cauchy ODE} also for functions $u$ satisfying the Osgood condition \cite{osgood1898beweis} or a uniform one-sided Lipschitz condition.} was obtained in \cite{diperna1989ordinary}, where the authors proved well-posedness for solutions of \eqref{continuity equation in Rd} in the class $L^\infty (0,T; L^q (\R^d))$ with a velocity field $u \in [L^1 (0,T; W^{1,p} (\R^d))]^d$ with $1/p + 1/q = 1$ and such that  $ \operatorname{div} (u) \in L^1 (0,T; L^\infty (\R^d))$ (see \cite{diperna1989ordinary} for the details). Well-posedness was later extended in \cite{ambrosio2004transport} to velocity fields $u \in [L^1 (0,T; BV(\R^d))]^d$ with the negative part of the divergence (denoted in the following as $\div(u)^-$) in $L^1 (0,T; L^\infty (\R^d))$. Moreover, in \cite{ambrosio2004transport} the author introduced the following notion of ``regular Lagrangian flow'' which generalizes the classical notion of flow.

\begin{definition}{\textbf{(Regular Lagrangian Flow, \cite{ambrosio2004transport})}}\label{rlf definition}

We say that $\Phi$ is a regular Lagrangian flow associated to the vector field $u$ if:

\begin{itemize}
    \item For $\Le ^d$ a.e. $x \in \R ^d$ , $t \mapsto \Phi(t,x)$ is an absolutely continuous solution of $\partial_t \Phi(t,x) = u(t,\Phi(t,x))$ with $\Phi(0,x)=x$;
    \item Let $\Le^d$ be the $d$-dimensional Lebesgue measure. There exists a constant $L=L(T)>0$ such that $\Phi(t, \cdot )_\# \Le ^d \leq L \Le ^d $ for every $t \in [0,T]$.
\end{itemize}

\end{definition}

The constant $L$ is often called in the literature ``compressibility constant'' and it can be proved that this constant is essentially controlled by $\div (u)^-$ (in fluid dynamics a common condition on velocity fields is the incompressibility, namely $\div(u)=0$, which in turn implies $L=1$). Even though the authors proved the well-posedness, both in \cite{diperna1989ordinary} and in \cite{ambrosio2004transport} they did not exhibit any kind of quantitative stability, since their proofs rely on the method of renormalized solutions, which is not well suited for quantitative estimates. The first quantitative result was proved in \cite{crippa2008estimates}, where the authors derive stability estimates on regular Lagrangian flows for $u \in [L^1(0,T; W^{1,p} (\R^d)]^d$ with $p >1$ \footnote{The key estimate in their paper relies on the $L^p -L^p$ strong boundedness of the maximal function, which only holds if $p >1$.}. More recently in \cite{seis2017quantitative, optimalcontinuitySeis} the author used tools from optimal transport theory and the results from \cite{crippa2008estimates} to give stability estimates in a suitable logarithmic Wasserstein distance for the solution $\rho$ of \eqref{continuity equation in Rd}.

\vspace{20pt}

In recent years the topic of approximating $\rho$ in presence of non Lipschitz velocity fields with numerical methods has also been investigated. To the best of our knowledge, the first works in this direction have been \cite{walkington2005convergence} and \cite{boyer2012analysis}, where using a discontinuous Galerkin method and a finite volume scheme respectively they prove strong convergence of approximate solution, although without explicit rates of convergence. However, it is worth mentioning that in \cite{boyer2012analysis} they conjecture a rate of $1/2$ based on numerical simulations, which matches the rate for smoother vector fields \footnote{The method is actually of first order, but the rate of convergence falls to 1/2 if the initial datum is not regular, see \cite{Kuznecov1979,DespresLaxtheorem2004,boyer2012analysis}.}. The rate was proved to be optimal, although only for Lipschitz velocity field, in \cite{MerletLinftyL22007,MerletErrorEstimate2007}. The first quantitative estimate of the error in presence of a vector field with Sobolev regularity (although only in a weak sense, i.e. the logarithmic Wasserstein distance introduced in this context in \cite{seis2017quantitative}) was obtained in \cite{schlichting2017convergence}, where the authors consider an \textit{explicit} finite volume upwind scheme and whose proof requires (for technical reasons) the use of Cartesian meshes \footnote{I.e. a mesh whose elements are isometric axis-parallel rectangular boxes.} to discretize the spatial domain. Such results have been extended to more general meshes in \cite{schlichting2018analysis} by using an \textit{implicit} method, and recently they have been adapted to include diffusion in \cite{navarro2023error, navarro2022optimal}. Recent developments concerning the numerical approximation of \eqref{continuity equation in Rd} have been obtained in \cite{ben2019convergence} and \cite{jabin2024discretizing}. In \cite{ben2019convergence} in particular they obtain strong convergence (i.e. in $L^p$ norm) of approximate solutions of non linear continuity equations of the form $\partial_t \rho + \div(u f(\rho)) =0$ with different approximation schemes: for example their results prove strong convergence of the Lax-Friedrichs (\cite{lax1954weaksol, LeVeque1992NumericalMethods, LeVeque2002Finitevolume}) and of the upwind scheme. Their argument relies on bounding the semi-norms of the approximate solutions on a precise logarithmic Sobolev space (see also \cite{Belg2013compactness}) which embeds compactly in $L^p$. We refer the interested reader to \cite[Theorem 8]{ben2019convergence} pointing out that, while not explicitly derived, the results obtained in \cite{ben2019convergence} allow to get a quantitative rate of convergence (see \cite[Corollary 11]{ben2019convergence} and the discussion after that). However, as in \cite{schlichting2017convergence}, the main restriction is that the estimates are only valid in a Cartesian mesh. The issue has been partially solved in \cite{jabin2024discretizing}, where the authors obtain similar results for the continuity equation on non Cartesian meshes with some additional hypotheses (they require some periodicity condition). Moreover, in \cite{jabin2024discretizing} they manage to treat a system of equations where the density solving \eqref{continuity equation in Rd} and the velocity are coupled in a non linear way (see \cite[Theorem 1.7]{jabin2024discretizing}). We stress that, since the continuity equation is often part of many different physical models, studying the coupling of \eqref{continuity equation in Rd} with other equations is crucial in many different parts of physics and biology. A non exhaustive list comprehend \cite{KineticenergycontrolGallouette2010,Gallouetunconditionallystable2008} for compressible Navier-Stokes, \cite{MACSchemeGallouet2023,HerbinConvergentFEMStokes2009II,HerbinConvergentFEMStokes2009I} for compressible Stokes, \cite{TopazSwarmingPatterns2004} for swarming and \cite{BurgerAsymptoticAnalysis2008} for chemotaxis, see also \cite{jabin2024discretizing} and the references therein. To the best of our knowledge, so far every approximation scheme has used an Eulerian point of view of \eqref{continuity equation in Rd}. We will instead use a Lagrangian one, managing to prove quantitative estimates, although only in a weak sense, for non Cartesian meshes without \textit{any} regularity assumptions, requiring only a bound on the diameter of the cells. This seems to be the first \textit{explicit} method which gives quantitative estimates on a non Cartesian mesh (we recall that in \cite{schlichting2018analysis} they used an \textit{implicit} scheme). Other useful features of our method, from the point of view of numerics, are that it is parallelisable in a straightforward way and that it does not need to satisfy a CFL condition (see Remark \ref{rmk: no cfl}).

\vspace{10pt}

The main result of this paper, i.e. Theorem \ref{theorem on estimate between real rlf and euler flow}, is a stability estimate in the spirit of \cite{crippa2008estimates}, which we use as a tool to prove approximation results for the solution of the continuity equation. In particular we consider the following initial value problem:

\begin{equation}\label{continuity equation}
    \begin{cases}
        \partial_t \rho + \div_x (u \rho)=0 &\text{ in } (0,T) \times \T^d \\
        \rho(0, \cdot )= \rho _0 &\text{ in }  \T^d .
    \end{cases}
\end{equation}

\begin{remark}\label{torus remark}
Our results are stated on the $d$-dimensional torus $\T^d$ for simplicity. However, with some care, they can be extended to more general ambient spaces. 
\end{remark}

We will work with functions belonging to the following class:

\begin{equation}\label{main hypotheses}
\begin{cases}
\rho_0 \in L^q (\mathbb{T}^d)\\
u \in [L^1 (0,T; W^{1,p} (\T ^d)) \cap L^\infty ((0,T) \times \T ^ d)]^d &\text{with }1/p + 1/q = 1 , \; p>1 \\
\div (u)^-  \in L^\infty ((0,T) \times \T^d) &\text{if } 1< p \leq d\\
\div (u)^-  \in L^1 (0,T ; L^\infty( \T^d)) &\text{if } p>d. 
\end{cases}
\end{equation}

Notice that when $ 1< p \leq d$ we are asking a little more than what is asked in \cite{ambrosio2004transport} for having well-posedness. Namely, we are not asking that $\div (u)^{-} \in L^1 (0,T; L^\infty (\T^d))$, but rather $\div (u)^-  \in L^\infty ((0,T) \times \T^d)$. This stronger assumption will be crucial in our proof of Lemma \ref{compressibility estimate}. The idea is to use Lagrangian estimates to get estimates at the Eulerian level (i.e. for the solution $\rho$ of \eqref{continuity equation}). The key observation is that in the DiPerna-Lions setting it still holds that the solution $\rho$ of $\eqref{continuity equation}$ can be represented as 

$$ \rho(t, \cdot) = \Phi (t, \cdot)_\# \rho_0,$$

with $\rho_0$ the initial datum and $\Phi$ the regular Lagrangian flow associated to the velocity $u$, as defined in Definition \ref{rlf definition}. Then we approximate the solution $\rho$ with

$$ \rho_E (t, \cdot) = \Phi_E (t, \cdot)_\# \mu , $$

where $\Phi_E$ is an approximation of $\Phi$ and $\mu$ is an approximation of $\rho_0$. This method is in principle quite flexible, in the sense that any ``good'' approximation of $\Phi$ and $\rho_0$ should give a ``good'' approximation of $\rho$ which we will specify more precisely later on. In particular we will approximate $\rho$ in two ways, using different approximations for both $\Phi$ and $\rho_0$. In the first one we consider a mesh of the domain with some mild assumptions on the sets $\{Q_i\}_{i=1} ^N$ of the mesh, and for every set $Q_i$ we choose (according to a uniform probability distribution) a point $x_i \in Q_i$ and a weight $M_i >0$ so that $\mu = \sum_{i=1} ^M M_i \delta_{x_i}$ approximates the initial datum $\rho_0$. As for the approximation of the regular Lagrangian flow $\Phi$ we construct, for every point $x \in \T^d$, a piecewise affine map $\Phi_E (\cdot , x)$ obtained by a (sort of) explicit Euler method applied to the velocity field $u$. The precise definition of $\Phi_E$ will be given in Definition \ref{definition of euler flow}. The approximation $\rho_E$ will therefore be a weighted sum of Dirac deltas whose positions evolve in time in a piecewise linear fashion. We prove that $\rho_E$ approximates in Wasserstein distance the solution $\rho$ of \eqref{continuity equation}, although only in probability (see Theorem \ref{1 wasserstein theorem}, Theorem \ref{logarithmic wasserstein theorem} and Proposition \ref{Monte-Carlo proposition} for the precise statements). We do not know if, at least in some cases, similar estimates hold for \emph{every} choice of the initial positions of the Dirac deltas, i.e. in a deterministic sense (see Remark \ref{remark that maybe estimates hold pointwise}). We also consider a different approximation $\bar{\Phi}_E$ of the flow (see Definition \ref{definition of mean euler flow}) which, if $p>d$ and considering an approximation of the initial datum in the strong $L^1$ norm, provides a deterministic approximation of $\rho$ in Wasserstein distance (see Theorem \ref{teorema approssimazione diffusa}). The novelty is that this an \textit{explicit} method which can be performed on unstructured meshes, it provides an easily parallelisable algorithm and it uses the Lagrangian formulation of the problem, whereas in previous works the Eulerian point of view was always employed. 

\vspace{20pt}

Let us now discuss briefly the organisation of the paper. Section 2 is dedicated to proving Theorem \ref{theorem on estimate between real rlf and euler flow}, where we give a novel $ L^\infty _t (L^p _x)$ stability estimate for the difference of the regular Lagrangian flow $\Phi$ of a velocity field $u$ satisfying \eqref{main hypotheses} and an approximation $\Phi_E$ of such flow, whereas previously known stability estimates considered the difference of two regular Lagrangian flows (see \cite{crippa2008estimates}). In Section 3 we use such a stability estimate to approximate, in Wasserstein distance, the solution $\rho$ of the continuity equation \eqref{continuity equation} via the push-forward through $\Phi_E$ of suitably chosen Dirac deltas. The estimates are proved both in $1$-Wasserstein distance and logarithmic Wasserstein distance. In the latter case in particular we show that in expected value the rate of convergence is of order $1/2$ if $1<p \leq d$, as the rate obtained in \cite{schlichting2017convergence}, and it is better than that if $p>d$ (see Theorem \ref{1 wasserstein theorem} and Theorem \ref{logarithmic wasserstein theorem} for the details). We believe that the logarithmic rates we obtain are optimal for the range $1 < p \leq d$ (as suggested also in \cite{schlichting2017convergence}), but the case $p>d$ does not give any precise rate. In Section 4 we consider, for reasons explained in Remark \ref{remark p>d}, only velocity fields $u\in [L^1 (0,T; W^{1,p} (\T^d) \cap L^\infty ((0,T) \times \T^d)] ^d$ such that $p>d$ and we assume to approximate the initial datum strongly in $L^1$. Under these stronger assumptions, using a slightly different approximation $\bar{\Phi}_E$ of the regular Lagrangian flow, we are able to prove estimates similar to the ones of Section 3, but which hold deterministically. In this case, moreover, we prove a convergence of order 1 of the method with respect to the logarithmic Wasserstein distance. We refer the reader to Theorem \ref{teorema approssimazione diffusa} for the details. In Section 5 we sum up the results obtained and highlight some open questions and future research perspectives. As a final note, we remark that in \cite{erroranalysistheta} we are also considering an approximation via a $\theta$-method of the regular Lagrangian flow, with which we can approximate the solution of \eqref{continuity equation} with a divergence free vector field which, even outside the range of DiPerna-Lions theory, is able to select the unique Lagrangian solution of \eqref{continuity equation}, i.e. the unique solution which is representable as a push-forward of the initial datum via the regular Lagrangian flow of the velocity field. 

\section{Stability estimates}

 Let us immediately state the main result of the paper, i.e. a stability estimate between the regular Lagrangian flow $\Phi$ and an approximation $\Phi_E$ of such flow, which will be introduced shortly after the statement.

\begin{theorem}{\textbf{($L^p$ stability estimate between $\Phi$ and $\Phi_E$)}} \label{theorem on estimate between real rlf and euler flow}

    Let $u$ satisfy assumptions \eqref{main hypotheses}. Then if $\Phi$ is the regular Lagrangian flow of $u$ and $\Phi_E$ is defined as in Definition \ref{definition of euler flow} with $\delta=\sqrt{\Delta t}$ if $1 < p \leq d$, we have

    $$ \|\Phi(t,\cdot) - \Phi_E (t,\cdot)\|_{L^p(\T^d)} \lesssim C_t |\log(\Delta t)|^{-1}, $$

    where $C_t \to 0$ as $t \to 0$. Moreover, if $p >d$ the same estimate holds even if $\div (u)^- \in L^1 (0,T; L^p (\T^d))$.
\end{theorem}

We point out here that we will make extensive use of the notation $X \lesssim Y$, which means $X \leq C Y$ for some positive constant $C>0$ that we omit explicitly writing both to keep the notation lighter and because $C$ is uniformly bounded. Crucially relying on Theorem \ref{theorem on estimate between real rlf and euler flow} we will be able to prove two approximation results for the solution $\rho$ of \eqref{continuity equation} in different Wasserstein distances: Theorem \ref{1 wasserstein theorem} and Theorem \ref{logarithmic wasserstein theorem}. In section 4 we will also introduce another approximation $\bar{\Phi}_E$ of $\Phi$ (see Definition \ref{definition of mean euler flow}) and we will prove a result analogous to Theorem \ref{theorem on estimate between real rlf and euler flow}, i.e. Proposition \ref{estimate between Phi and  bar PhiE}, with which we will get another approximation of the density $\rho$: see Theorem \ref{teorema approssimazione diffusa}.

\vspace{20pt}

The vector field $\Phi_E$ is explicitly constructed by a sort of forward Euler method, and we now give its precise definition.

\begin{definition}{\textbf{(Euler approximation of a regular Lagrangian flow)}}\label{definition of euler flow}

Let $u$ satisfy conditions \eqref{main hypotheses} with $1 < p \leq d$, and let 

\begin{equation}\label{regularised velocity} u_\delta (t, \cdot ) \coloneqq \eta_\delta \ast u (t, \cdot )
\end{equation}

for some $\delta >0$, where $\eta$ is a standard mollifier in $\R^d$, $\eta_\delta(x) \coloneqq \delta^{-d} \eta(x/\delta)$ and the convolution is in space only.  We denote as $\Phi_E$ (where  ``E" stands for ``Euler") the so called Euler approximation of the regular Lagrangian flow $\Phi$ of $u$, with time step $\Delta t$ and regularisation parameter $\delta$. For $ t \in (t_n , t_{n+1})$ we define:

\begin{equation}\label{euler flow definition equation}
    \Phi_E (t,x) \coloneqq \Phi_E (t_n, x) + (t-t_n) \fint_{t_n} ^{t_{n+1}} u_\delta(s, \Phi_E (t_n, x)) ds,
\end{equation}

where $0=t_0 < t_1 < \dots < t_M =T$ is a partition of $[0,T]$ with $t_{i+1} - t_i =\Delta t$ and \newline $\Phi_E (0,x)= x \; \forall x \in \T^d$. If $p>d$, we define $\Phi_E$ directly with the $u$ as 

\begin{equation}\label{euler flow definition equation p>d}
    \Phi_E (t,x) \coloneqq \Phi_E (t_n, x) + (t-t_n) \fint_{t_n} ^{t_{n+1}} u (s, \Phi_E (t_n, x)) ds.
\end{equation}

 We remark that in order to keep the notation as light as possible, we omit explicitly writing the dependence of $\Phi_E$ on $\Delta t$ and (eventually) $\delta$.

\end{definition}

In order to prove Theorem \ref{theorem on estimate between real rlf and euler flow}, we need some preliminary estimates. We now give a Lemma which estimates the intrinsic error in considering $u_\delta$ instead of $u$ if $1 < p \leq d$. Notice that this is a $L^p$ extension of an equivalent $L^1$ estimate obtained in \cite[Theorem 2.9]{crippa2008estimates} which may be of independent interest. To highlight this link, we have decided to keep the notation of \cite{crippa2008estimates}, denoting as $b$ and $\tilde{b}$ the vector fields and as $X$ and $\tilde{X}$ their regular Lagrangian flows.

\begin{proposition}{\textbf{($L^p$ stability estimate for regular Lagrangian flows)}}\label{Lp stability estimate for regular Lagrangian flows}

Let $b$ and $\tilde{b}$ be bounded vector fields belonging to $L^1 (0,T; W^{1,p} (\R^d))$ for some $p>1$. Let $X$ and $\tilde{X}$ be regular Lagrangian flows of $b$ and $\tilde{b}$ respectively and denote by $L$ and $\tilde{L}$ their compressibility constants. Then, for every time $t \in [0,T]$, we have

    $$ \|X( t, \cdot )- \tilde{X} (t, \cdot)\|_{L^p(B_r (0))} \leq C_t \left | \log (\| b - \tilde{b} \|_{L^1 (0,t; L^p ( B_R (0)))} ) \right|^{-1} ,$$

where $R= r + T\| \tilde{b}\|_\infty$ and the constant $C_t$ only depends on $t, r, \|b\|_\infty , \|\tilde{b}\|_\infty , L, \tilde{L} $ and $\|D_x b\|_{L^1 (L^p)}$. Moreover, $C_t \to 0$ as $t \to 0$.
 
\end{proposition}

\begin{proof}
Let $\delta \coloneqq \sqrt{\|b - \tilde{b} \|_{L^1 (0, t; L^p (B_{R} (0)))}} $ with $R= r + T\|\tilde{b} \|_\infty$ and let 

$$ g (t) \coloneqq \left[\int_{B_r (0)} \log \left(1+ \frac{|X (t,x) - \tilde{X} (t,x)|}{\delta} \right)^p dx \right]^{1/p}.$$

For simplicity of notation, we will omit explicitly writing the dependence on $t$ and $x$ when this will not be necessary. So we will often write $X $ and $ \tilde{X} $ in place of $X(t,x)$ and $ \tilde{X} (t,x)$. Differentiating in time:

\begingroup
\allowdisplaybreaks
\begin{align}\label{g' inequality}
    &\frac{d}{dt} g (t) \\
    &=\frac{1}{p} g(t) ^{1-p} \int_{B_r (0)} p \log \left( 1 + \frac{|X - \tilde{X}|}{\delta} \right) ^{p-1} \frac{\delta }{\delta + |X - \tilde{X}|}  \frac{1}{\delta} \frac{X - \tilde{X}}{|X - \tilde{X}|} \left(\frac{d}{dt}X - \frac{d}{dt} \tilde{X} \right)dx \nonumber \\
    &\leq \frac{g(t)^{1-p}}{\delta}  \int_{B_r (0)} \log \left(  1 + \frac{|X - \tilde{X}|}{\delta} \right) ^{p-1} \frac{1}{1 + |X - \tilde{X}|/\delta} |b (t, X) - \tilde{b} (t, \tilde{X})| dx  \nonumber \\
    & \leq \underbrace{\frac{g(t)^{1-p}}{\delta}  \int_{B_r (0)} \log \left(  1 + \frac{|X - \tilde{X}|}{\delta} \right) ^{p-1} \frac{1}{1 + |X - \tilde{X}|/\delta} |b (t, X) - b (t, \tilde{X})| dx}_{\circled{1}}  \nonumber \\
    &+  \underbrace{\frac{g(t)^{1-p}}{\delta}  \int_{B_r (0)} \log \left(  1 + \frac{|X - \tilde{X}|}{\delta} \right) ^{p-1} \frac{1}{1 + |X - \tilde{X}|/\delta} |b (t, \tilde{X}) - \tilde{b} (t, \tilde{X})| dx}_{\circled{2}}. \nonumber 
\end{align} 
\endgroup

Considering $\circled{1}$ we have, thanks to Lemma \ref{pointwise bv inequality}:

\begin{align}\label{velocity pointwise inequality}
 |b (t, X (t,x)) - b (t, \tilde{X} (t,x))| \leq c_n |X (t,x) - \tilde{X} (t,x)| &[M_{\tilde{R}} (|Db|) (t, X(t,x)) \\
 &+ M_{\tilde{R}} (|Db|) (t, \tilde{X}(t,x))],   \nonumber
 \end{align}

 for almost every $t \in [0, T]$ and (at fixed $t$) for almost every $x \in \R^d$, with $\tilde{R} = T(\|b\|_\infty + \|\tilde{b}\|_\infty)$.

\begin{remark}

Notice that the compressibility constraint in the definition of regular Lagrangian flow is crucial for \eqref{velocity pointwise inequality} to hold almost everywhere. Indeed, for almost every $t$ we have that $b_t (x) \coloneqq b (t,x) \in W^{1, p} (\R^d)$. Then \eqref{velocity pointwise inequality} holds only if $X(t,x) , \tilde{X} (t,x) \in \R^d \setminus N_t$, with $N_t$ a null set. Then inequality \eqref{velocity pointwise inequality} holds for almost every $x \in \R^d$ since

$$ \Le ^d (X(t, \cdot ) ^{-1} (N_t)) \leq L \Le ^d (N_t) =0 \text{ and } \Le ^d (\tilde{X} (t, \cdot ) ^{-1} (N_t)) \leq \tilde{L} \Le ^d (N_t) =0.$$
\end{remark}

Using \eqref{velocity pointwise inequality}:

\allowdisplaybreaks
\begin{align*}
    &\begin{aligned}[t]
    \circled{1} \leq c_n g(t) ^{1-p} \int_{B_r (0)} \log \left(  1 + \frac{|X - \tilde{X}|}{\delta} \right) ^{p-1} \frac{|X - \tilde{X}|/\delta}{1 + |X - \tilde{X}|/\delta}&[M_{\tilde{R}}(| Db|) (t, X(t,x)) \\
    & \hspace{-1cm}+M_{\tilde{R}} (|Db|) (t, \tilde{X}(t,x))] dx
    \end{aligned}
    \\
    &\begin{aligned}
    \leq  c_n g(t) ^{1-p} \int_{B_r (0)} \log \left(  1 + \frac{|X - \tilde{X}|}{\delta} \right) ^{p-1}  &[M_{\tilde{R}} (|Db|) (t, X(t,x))\\
    &+ M_{\tilde{R}} (|Db|) (t, \tilde{X}(t,x))] dx
    \end{aligned}\\
    &=\underbrace{c_n g(t) ^{1-p} \int_{B_r (0)} \log \left(  1 + \frac{|X - \tilde{X}|}{\delta} \right) ^{p-1} M_{\tilde{R}} (|Db|) (t, X(t,x)) dx}_{ \circled{3}}  \\
    &+\underbrace{c_n g(t) ^{1-p} \int_{B_r (0)} \log \left(  1 + \frac{|X - \tilde{X}|}{\delta} \right) ^{p-1} M_{\tilde{R}} (|Db|) (t, \tilde{X}(t,x)) dx}_{\circled{4}}.
\end{align*}

For $\circled{3}$, applying H\"older's inequality with exponents $p$ and $p/(p-1)$ yields

$$ \circled{3} \leq c_n g(t) ^{1-p} g(t) ^{p-1} \left( \int_{B_r (0)} [M_{\tilde{R}} Db (t, X(t,x))]^p dx \right) ^{1/p},$$

and changing variables setting $y=X(t,x)$ gives

$$ \circled{3} \leq c_n L^{1/p} \left(\int_{B_{r+ T \|b\|_\infty} (0)} [M_{\tilde{R}} Db (t,y)]^p dy \right)^{1/p} \leq c_n c_{p,n} L^{1/p} \|Db (t, \cdot)\|_{L^p (B_{R'} (0))}, $$

with $R'= r + 3 T \max\{\|b\|_\infty, \|\tilde{b}\|_\infty \} $, thanks to the continuity of the local maximal function on $L^p$ (\cite[Lemma A.2]{crippa2008estimates}). With this choice of $R'$, it is easy to see that with the same argument we get

$$ \circled{4} \leq c_n c_{p,n} \tilde{L}^{1/p} \|Db (t, \cdot)\|_{L^p (B_{R'} (0))} .$$

As for $\circled{2}$, we clearly have:

\begin{align*}
    \circled{2} \leq \frac{g(t) ^{1-p}}{\delta} \int_{B_r (0)} \log \left(  1 + \frac{|X - \tilde{X}|}{\delta} \right) ^{p-1}|b (t, \tilde{X} (t,x)) - \tilde{b} (t, \tilde{X} (t,x))| dx.
\end{align*}

First using H\"older's inequality and then changing variables as before:

\begin{align*}
    \circled{2} \leq \frac{ \tilde{L} ^{1/p} }{\delta } \|b(t, \cdot) - \tilde{b} (t, \cdot) \|_{L^p (B_{R} (0))}    ,
\end{align*}

where we recall that $R= r + T \|\tilde{b}\|_\infty$. Inequality \eqref{g' inequality} now reads:

$$ \frac{d}{dt} g(t) \leq c_n c_{p,n} \left( L^{1/p} + \tilde{L} ^{1/p} \right)  \|Db (t, \cdot)\|_{L^p (B_{R'} (0))}  + \frac{\tilde{L} ^{1/p}}{\delta} \|b(t, \cdot) - \tilde{b} (t, \cdot) \|_{L^p (B_{R} (0))} .$$

Thanks to the choice of $\delta$, integrating in time from $0$ to $t \in [0,T]$ we get:

$$ g(t) \leq C_t \text{ for all } t \in [0, T],$$

with $C_t$ being a constant depending on $ t, r, L, \tilde{L}, \|D_x b \|_{L^1 (L^p (B_{R'} (0))} , \|b\|_\infty , \|\tilde{b}\|_\infty $ and $p$, and it can be easily seen that $C_t \to 0$ as $t \to 0$. This means that, for every $t \in [0,T]$:

\begin{equation}\label{Boundedness of the Lp norm of the logarithm} \int_{B_r (0)} \log \left(1+ \frac{|X (t,x) - \tilde{X} (t,x)|}{\delta} \right)^p dx  \leq C_t ^p .\end{equation}
 
By Chebychev inequality, for every $\eta>0$ there exists a set $K$ such that $|B_r \setminus K| \leq \eta$ and 
 
 $$ \log \left( \frac{|X (t,x) - \tilde{X} (t,x)|}{\delta} +1\right) ^p \leq \frac{C_t ^p}{\eta} \text{ on }K. $$ 
 
 Then, on $K$:

 $$ |X (t,x) - \tilde{X} (t,x)|^p \leq \delta^p \exp \left(\frac{C_t p}{\eta^{1/p}} \right).$$

 Integrating over $B_r$:

\begin{equation} \label{mid Lp estimate}
\int_{B_r (0)} |X (t,x)- \tilde{X} (t,x) |^p dx \leq \eta (\|X \|_\infty + \|\tilde{X} \|_\infty ) ^p + \omega_d r^d \delta^p \exp \left(\frac{C_t p}{\eta^{1/p}} \right),
\end{equation}

with $\omega _d \coloneqq |B_1 (0)|$. Choosing $\eta= 2^p C_t ^p |\log \delta |^{-p}$ we get

$$\exp \left(\frac{C_t p}{\eta^{1/p}} \right)= \exp\left(\frac{p}{2|\log(\delta)|^{-1}}\right)= \exp \left( - \frac{p}{2} \log(\delta) \right)= \exp( \log(\delta^{-p/2})) = \delta ^{-p/2} ,$$

so \eqref{mid Lp estimate} reads as:

\begin{equation}\label{last estimate}   
\int_{B_r (0)} |X(t,x) - \tilde{X} (t,x) |^p 
 dx \leq 2^p C_t ^p (\|X \|_\infty + \|\tilde{X} \|_\infty )^p  |\log \delta |^{-p}  + \omega_d r^d \delta ^{p/2} .
 \end{equation} 

Taking the p-th root, and assuming that $\delta= \sqrt{\| b - \tilde{b} \|_{L^1 (0, t; L^p ( B_R (0)))}}$ is small enough so that $|\log(\delta)|^{-1} >>\sqrt{\delta}$, we conclude.
    \end{proof}

\begin{remark}[The choice of $\delta$]\label{remark on boundedness of logarithm Lp estimate with Lp norm of difference velocities}
\

    In the previous Proposition we assumed $\delta \coloneqq \sqrt{\|b - \tilde{b} \|_{L^1 (0, t; L^p (B_{R} (0)))}} $ in order to bound $g(t)$ with a constant $C_t \to 0$ as $t \to 0$. If instead we had chosen $\delta =\|b - \tilde{b} \|_{L^1 (0, t; L^p (B_{R} (0)))}$ the proof would have still worked, with $g(t) \leq C$. In particular, it holds:

    \begin{equation*}
    \left| \left| \log \left( 1 + \frac{|X(t, \cdot) - \tilde{X}(t, \cdot)|}{\|b - \tilde{b} \|_{L^1 (0, t; L^p (B_{R} (0)))}} \right) \right| \right|_{L^p (\T^d)} \leq C.
    \end{equation*}
\end{remark}

Thanks to Proposition \ref{Lp stability estimate for regular Lagrangian flows} we easily obtain:

\begin{corollary}{\textbf{(Error of regularisation)}}\label{Error of regularisation}

    Consider $\Phi$ the regular Lagrangian flow of a vector field $u$ satisfying \eqref{main hypotheses}, and let $\Phi_\delta$ be the (classical) flow of $u_\delta$ defined in \eqref{regularised velocity}. Then for every $t \in [0,T]$:

$$\|\Phi (t, \cdot) -\Phi_\delta(t, \cdot)\|_{L^p (\T^d)} \lesssim  C_t |\log (\|u- u_\delta \|_{L^1(0,t; L^p (\T^d))})|^{-1} \leq  C_t |\log (\delta)|^{-1} ,$$

with $C_t \to 0$ as $t \to 0$.
 
\end{corollary}

\begin{proof}
For the first inequality, use the $L^p$ estimate between regular Lagrangian flows in Proposition \ref{Lp stability estimate for regular Lagrangian flows}. For the second one notice that denoting with $B_1$ the ball centered in $0$ and of unitary radius, using Lemma \ref{pointwise bv inequality}:

    \begin{align*}
        &\left|u(t,x) -  u_\delta (t,x) \right| = \left| u(t,x) - \int_{\T^d} \eta_\delta (z) u (t, x+z) dz   \right|\\
        &= \left| u(t,x) - \int_{B_1} \eta(y) u(t, x+ \delta y) dy \right|  \leq \int_{B_1} \eta(y)|u(t,x) - u(t, x+ \delta y)| dy  \\
        &\lesssim \delta \int_{B_1} \eta(y) [M(|Du|)(t,x) + M(|Du|)(t, x+ \delta y) ] dy \nonumber \\ 
        & = \delta M(|Du|)(t,x) + \delta [\eta_\delta \ast M(|Du|) ](t,x).\nonumber
    \end{align*}

    Take the $L^p$ norm on both sides and  conclude using the Young convolution inequality and the boundedness of the maximal operator, and finally integrating in time we have

    \begin{align*}
     &\|u - u_\delta \|_{L^1 (0,t; L^p(\T^d))} \lesssim \delta \implies
    \log(\|u - u_\delta \|_{L^1 (0,t; L^p(\T^d))}) \lesssim \log(\delta)\\
    & \implies | \log(\|u - u_\delta \|_{L^1 (0,t; L^p(\T^d))}) | \gtrsim |\log(\delta)|\\
    &\implies  | \log(\|u - u_\delta \|_{L^1 (0,t; L^p(\T^d))} )|^{-1} \lesssim |\log(\delta)|^{-1}.
     \end{align*}
    \end{proof}

We now need to estimate $\|\Phi_\delta (t, \cdot)- \Phi_E (t, \cdot)\|_{L^p (\T^d)} $ uniformly in $t$, but since $\Phi_E$ is not properly a flow we cannot use the stability estimate in Proposition \ref{Lp stability estimate for regular Lagrangian flows}. We need the following Propositions as a key tool to treat the case $1 < p \leq d$. As we will see later, this will not be needed if $p>d$.

\begin{proposition}{\textbf{(Compressibility estimate for $\Phi_E$)}}\label{compressibility estimate}

Consider $\Phi_E$ defined as in \eqref{euler flow definition equation} with $\delta = \sqrt{\Delta t}$. Then $\exists c>0$ not dependent on $\Delta t$ such that $\det (\nabla_x \Phi_E (t,x)) >c$ for every $(t,x) \in [0,T] \times \T^d$.
\end{proposition}

\begin{proof}

    Define $\Phi_{n} (t,x)$ in $[0,\Delta t] \times \T^d$ as

    $$ \Phi_{n} (t,x) \coloneqq x + t \fint _{t_n} ^{t_{n+1}} u_\delta (s, x) ds.$$

    Notice that $\Phi_E$ can be seen as a composition of such $\Phi_{n}$'s, i.e. if $ t \in [t_n , t_{n+1}]$:

    $$\Phi_E (t,x)= \Phi_{n} (t-t_n, \Phi_{n-1} (\Delta t, \Phi_{n-2} (\dots , \Phi_{0} (\Delta t,x))) \dots ). $$

    Differentiating in space:

    $$ \nabla_x \Phi_E (t,x)= \nabla_x \Phi_{n} (t-t_n , \dots) \nabla_x \Phi_{n-1} (\Delta t, \dots) \dots \nabla_x \Phi_{0} (\Delta t, x),$$
    
    and by Binet:

    $$  \det[\nabla_x \Phi_E (t,x)]= \underbrace{\det[\nabla_x \Phi_{n} (t-t_n , \dots)] \det[\nabla_x \Phi_{n-1} (\Delta t, \dots)] \dots \det[\nabla_x \Phi_{0} (\Delta t, x)]}_{\approx 1/\Delta t \text{ times}}.$$
  
    We will uniformly bound $\det [ \nabla_x \Phi_{0}]$ from below, but the argument is the same for every $\Phi_j$ (the only difference in the definition is the time interval). We have

    \begin{align*} 
    \det\left(\nabla _x \Phi_{0} (\Delta t,x) \right) &= \det \left(I + \Delta t \fint_{0} ^{\Delta t} \nabla_x u_\delta (s,x) ds \right)\\
    &= \det (I + \Delta t A), \; \text{ where } A \coloneqq \fint_{0} ^{\Delta t} \nabla_x u_\delta (s,x) ds.
    \end{align*}

    It is known (see \cite[Chapitre I ,section 3 , equations (4) and (6)]{grothendieck1956theorie}) that

    $$ \det (I+ \Delta t A)= 1 + \Delta t [tr (A)] + \sum_{j=2} ^{d} (\Delta t)^j \sigma_j (\lambda_1 , \dots , \lambda_d),$$

    where $\sigma_j$ are symmetric polynomials of degree $j$ and $\lambda_i$ are the eigenvalues of $A$. By our definition each entry of the matrix $A$ is bounded by $\approx 1/ \sqrt{\Delta t}$ since $u_\delta$ is $1/ \sqrt{\Delta t}$-Lipschitz, and thus $|\lambda_i| \lesssim 1/ \sqrt{\Delta t}$ for every $i$. Then, $|\sigma_j (\lambda_1 , \dots, \lambda_d)| \lesssim (\Delta t)^{-j/2}$. Since we also have that $tr(A) ^- = \div(u_\delta) ^- \in L^\infty ((0,T) \times \T^d) $,  there exists $C>0$ such that:

    \begin{equation}\label{lower bound determinant on one of the factors} \det (I+ \Delta t A) \geq 1 - \Delta t [\div(u) ^{-}] - \sum_{j=2} ^{d} (\Delta t)^j (\Delta t)^{-j/2} \geq 1 - C\Delta t.
    \end{equation}

    Multiplying $\approx 1/\Delta t$ times such determinants we can estimate $\det (\nabla_x \Phi_E)$ uniformly from below, and the proof is complete. 
\end{proof}

\begin{remark}
    In the following we will always assume $\delta = \sqrt{\Delta t}$. In view of the proof of Proposition \ref{compressibility estimate} it is clear that other choices are possible. However, as we will see in Theorem \ref{logarithmic wasserstein theorem}, the choice of $\delta= \sqrt{\Delta t}$ will be crucial for obtaining a $O(\sqrt{\Delta t})$ rate of convergence if we consider a logarithmic Wasserstein distance, assuming that the space discretization scale $\Delta x$ is equal to $\Delta t$.
\end{remark}

\begin{proposition}{\textbf{(Injectivity and Lipschitzianity of $\Phi_E$)}}\label{Lipschitz and injective euler flow}

Let $\Phi_E$ be defined as in \eqref{euler flow definition equation} and let $\delta= \sqrt{\Delta t}$. Then $\exists C_{ \Delta t,T} >0$ such that

$$ C_{\Delta t, T}^{-1} |x-y| \leq |\Phi_E (t,x) - \Phi_E (t,y)| \leq C_{\Delta t, T} |x-y| \text { for all } t \in [0,T]. $$

 In particular, $\Phi_E (t, \cdot) $ is Lipschitz continuous and injective. 

\end{proposition}

\begin{proof}

  Consider two points $x, y \in \T^d$, with $|x - y| =\ell$. At the first step in the definition of $\Phi_E$ (i.e. up to time $t=\Delta t$) we have, since $u_\delta$ is $1/ \sqrt{\Delta t}$-Lipschitz:

        \begin{align*}
         &|\Phi_E (\Delta t, x) - \Phi_E (\Delta t, y)| \leq |x - y | +  \int_0 ^{\Delta t }|u_\delta (s, x) - u_\delta (s, y)| ds \lesssim \ell (1+\sqrt{\Delta t}).
    \end{align*}

    Iterating $T/\Delta t$ times we get a Lipschitz constant $C_{\Delta t,T} =\exp(T/\sqrt{\Delta t})$, and with the same argument we get an estimate from below with $ C_{\Delta t , T} ^{-1} =\exp(-T/\sqrt{\Delta t})$. 
\end{proof}

We are now ready to prove the following Proposition, whose proof is strongly inspired by the proof of \cite[Lemma 6]{schlichting2018analysis}:

\begin{proposition}{\textbf{($L^p$ stability estimate for $\Phi_\delta$ and $\Phi_E$)}}\label{proposition Lp estimate regularised rlf and euler flow}

Let $\Phi_E$ be defined as in \eqref{euler flow definition equation} with $u$ satisfying assumptions \eqref{main hypotheses}, and let $\delta = \sqrt{\Delta t}$. Then:

$$ \|\Phi_\delta (t,\cdot) - \Phi_E (t,\cdot)\|_{L^p (\T^d)} \lesssim C_t |\log(\Delta t)|^{-1} \text{ for every } t \in [0,T],$$

where $C_t \to 0$ as $t \to 0$ and $\Phi_\delta$ is the flow of $u_\delta$. If $u \in L^1 (0,T; W^{1,p} (\T^d )) $ with $p>d$, then

$$ \|\Phi(t, \cdot) - \Phi_E (t, \cdot)\|_{L^p (\T^d)} \lesssim C_t |\log(\Delta t)|^{-1} \text{ for every } t \in [0,T],$$

where $C_t \to 0 $ as $t \to 0$ and $\Phi_E$ is defined as in \eqref{euler flow definition equation p>d}. Moreover, if $p>d$ the conclusion holds even if we only require $\div (u) ^{-} \in L^1 (0,T; L^\infty (\T^d))$.
\end{proposition}

\begin{proof}

    In the first part of the proof we will work without loss of generality with $\Phi_\delta$ with the assumption $1 <p \leq d$, since the argument is identical if $p>d$. The difference will appear only at the end of the proof, where we differentiate the two cases. We want to prove that

    \begin{equation}\label{Lp norm of logarithm} \sup_{t \in [0,T]} \left | \left | \log \left( 1 + \frac{|\Phi_\delta (t,\cdot)- \Phi_E (t,\cdot)|}{\Delta t} \right) \right| \right|_{L^p (\T^d)} \leq C_t,
    \end{equation}

    with a constant $C_t$ not dependent on $\Delta t$ such that $C_t \to 0$ as $t \to 0$. Once we prove that, the conclusion follows as in Proposition \ref{Lp stability estimate for regular Lagrangian flows}. Notice that, by the concavity of the logarithm and by the fact that $\log(1+x)$ is an increasing monotone function, for every $a,b \geq 0$ it holds

    $$ \log (1+ b)- \log(1+a) \leq \frac{d}{dx}  \Big[\log(1+x) \Big]_{x=a} |b-a|= \frac{|b-a|}{1+a}.$$

   In particular, let $t=t_n$ for simplicity. We have (since the estimate is pointwise we omit the dependence on $x$ to keep the notation as light as possible):

    \begin{align*} 
    &\log \left( 1 + \frac{|\Phi_\delta (t_n)- \Phi_E (t_n)|}{\Delta t} \right) - \log \left( 1 + \frac{|\Phi_\delta (t_{n-1})- \Phi_E (t_{n-1})|}{\Delta t} \right) \\
    &\leq \frac{\Delta t}{|\Phi_\delta (t_{n-1}) - \Phi_E (t_{n-1})| + \Delta t } \frac{|\;|\Phi_\delta (t_n)- \Phi_E (t_n)| - |\Phi_\delta (t_{n-1}) - \Phi_E (t_{n-1})| \;|}{\Delta t}  \\
    &\leq \frac{|\Phi_\delta (t_n) - \Phi_\delta (t_{n-1}) -( \Phi_E (t_n) - \Phi_E (t_{n-1})) |}{|\Phi_\delta (t_{n-1}) - \Phi_E (t_{n-1}) | + \Delta t} \\
    &= \frac{|\int_{t_{n-1}} ^{t_n} u_\delta (s, \Phi_\delta (s)) - u_\delta (s, \Phi_E (t_{n-1})) ds |}{|\Phi_\delta (t_{n-1}) - \Phi_E (t_{n-1}) | + \Delta t}.
    \end{align*}

    Iterating the estimate above up to $t_0=0$ we get by telescopic summing:

    \allowdisplaybreaks
    \begin{align}\label{logarithmic estimate}
    &\log \left( 1 + \frac{|\Phi_\delta (t_n)- \Phi_E (t_n))|}{\Delta t} \right) \leq \sum_{j=0} ^{n-1} \frac{|\int_{t_{j}} ^{t_{j+1}} u_\delta (s, \Phi_\delta (s)) - u_\delta (s, \Phi_E (t_j)) ds|}{|\Phi_\delta (t_j) - \Phi_E (t_j)|  + \Delta t}   \\
    &\leq   \sum_{j=0} ^{n-1} \frac{|\int_{t_{j}} ^{t_{j+1}} u_\delta (s, \Phi_\delta (s)) - u_\delta (s, \Phi_\delta (t_j)) ds| + |\int_{t_{j}} ^{t_{j+1}} u_\delta (s, \Phi_\delta (t_j)) - u_\delta (s, \Phi_E (t_j)) ds|}{|\Phi_\delta (t_j) - \Phi_E (t_j)|  + \Delta t}  \nonumber\\
    &\leq \underbrace{\sum_{j=0} ^{n-1} \frac{\int_{t_{j}} ^{t_{j+1}} |u_\delta (s, \Phi_\delta (s)) - u_\delta (s, \Phi_\delta (t_j))| ds }{|\Phi_\delta (t_j) - \Phi_E (t_j)|  + \Delta t}}_{= \circled{1}} + \underbrace{\sum_{j=0} ^{n-1} \frac{\int_{t_{j}} ^{t_{j+1}} |u_\delta (s, \Phi_\delta (t_j)) - u_\delta (s, \Phi_E (t_j))| ds}{|\Phi_\delta (t_j) - \Phi_E (t_j)|  + \Delta t}}_{= \circled{2}} .\nonumber 
    \end{align}

Now, notice that by Lemma \ref{pointwise bv inequality}:

\begin{align*}
&\int_{t_{j}} ^{t_{j+1}} |u_\delta (s, \Phi_\delta (s)) - u_\delta (s, \Phi_\delta (t_j))| ds  \\
& \lesssim \int_{t_{j}} ^{t_{j+1}} |\Phi_\delta (s)- \Phi_\delta (t_j)| \;[M(|Du_\delta |)(s, \Phi_\delta (s)) + M(|Du_\delta |)(s,\Phi_\delta (t_j))] ds \\
&=  \int_{t_{j}} ^{t_{j+1}} \left|\int_{t_j} ^s u_\delta (s, \Phi_\delta (s)) ds \right| [M(|Du_\delta |)(s, \Phi_\delta (s)) + M(|Du_\delta |)(s,\Phi_\delta (t_j))] ds \\
&\leq \|u_\delta \|_\infty \Delta t  \int_{t_j} ^{t_{j+1}} [M(|Du_\delta |)(s, \Phi_\delta (s)) + M(|Du_\delta |)(s,\Phi_\delta (t_j))] ds.
\end{align*}

Then, since $\Delta t/ (|\Phi_\delta (t_j) - \Phi_E (t_j)|  + \Delta t) \leq 1$, we have

\begin{align*} 
\circled{1} &\lesssim \sum_{j=0} ^{n-1} \int_{t_{j}} ^{t_{j+1}} [M(|Du_\delta |)(s, \Phi_\delta (s)) + M(|Du_\delta |)(s,\Phi_\delta (t_j))] ds\\
&= \int_{0} ^{T} M(|Du|)(s, \Phi_\delta (s))ds + \sum_{j=0} ^{n-1} \int_{t_j} ^{t_{j+1}} M(|Du_\delta |)(s,\Phi_\delta (t_j)) ds.
\end{align*}

 Considering the $L^p$ norm:

\begin{align}\label{I don't know how to call you}
&\left| \left|\int_{0} ^{T} M(|Du|)(s, \Phi_\delta (s))ds + \sum_{j=0} ^{n-1} \int_{t_j} ^{t_{j+1}} M(|Du_\delta |)(s,\Phi_\delta (t_j)) ds \right| \right|_{L^p (\T^d)}   \\
& \leq \int_{0} ^{T} \|M(|Du_\delta |)(s, \Phi_\delta (s))\|_{L^p (\T^d)} ds +  \sum_{j=0} ^{n-1} \int_{t_j} ^{t_{j+1}} \|M(|Du_\delta |)(s,\Phi_\delta (t_j))\|_{L^p (\T^d)} ds  \nonumber \\
&\lesssim \|M(D u_\delta)\|_{L^1 (0,T;L^p (\T^d))} \lesssim \|Du\|_{L^1 (0,T; L^p (\T^d))}, \nonumber
\end{align}

where in the last inequalities we used the $L^\infty$ bound on the compressibility to change variables in each addendum, the boundedness of the maximal operator and Young's convolution inequality. As for $\circled{2}$, we need to differentiate the case $1 < p \leq d$ and $p>d$.

\begin{itemize}
    \item If $1<p \leq d$, we repeat the argument above. We have

    \begin{align*}
    &\int_{t_{j}} ^{t_{j+1}} |u_\delta (s, \Phi_\delta (t_j)) - u_\delta (s, \Phi_E (t_j))| ds \\
    &\lesssim |\Phi _h (t_j) - \Phi_E (t_j)|  \int_{t_j} ^{t_{j+1}} [M(|Du_\delta |)(s, \Phi_\delta (t_j)) + M(|Du_\delta |)(s,\Phi_E (t_j))] ds,
    \end{align*}

    and we can estimate $|\Phi _h (t_j) - \Phi_E (t_j)|/(|\Phi _h (t_j) - \Phi_E (t_j)| + \Delta t) \leq 1$. Then, after summing over $j$:

    $$ \circled{2} \lesssim \sum_{j=0} ^{n-1}  \int_{t_j} ^{t_{j+1}} [M(|Du_\delta |)(s, \Phi_\delta (t_j)) + M(|Du_\delta |)(s,\Phi_E (t_j))] ds.  $$
    
   After considering the $L^p$ norm the estimate goes as in \eqref{I don't know how to call you}. For the second addendum, notice that by Proposition \ref{Lipschitz and injective euler flow} the map $\Phi_E$ is Lipschitz and injective (so we are allowed to change variables in the usual way) and by Proposition \ref{compressibility estimate} the compressibility is controlled, thus the change of variables does not cause any problem as in \eqref{I don't know how to call you} when we were dealing with $\Phi_\delta$, and we can conclude.

    \item If $p >d$, recall that we work directly with $\Phi$, and not with $\Phi_\delta$. In this case we use Lemma \ref{lemma Caravenna Crippa} with $\tilde{p} \in (d, p)$ to get:

\begin{align} \label{estimate for p>d in the approximation}
\frac{\int_{t_{j}} ^{t_{j+1}} |u (s, \Phi (t_j)) - u (s, \Phi_E (t_j)) |ds}{|\Phi (t_j) - \Phi_E (t_j)|  + \Delta t} &\lesssim \frac{|\Phi (t_j) - \Phi_E (t_j)| \int_{t_{j}} ^{t_{j+1}} f(s, \Phi (t_j)) ds|}{|\Phi(t_j) - \Phi_E (t_j)|  + \Delta t} \\
&\leq \int_{t_j} ^{t_{j+1}} f (s, \Phi(t_j)) ds , \nonumber 
\end{align}

where 

\begin{equation}\label{definition of f for the estimate in p>d}
f (s, x) \coloneqq [M(|Du |^{\tilde{p}}) (s,x)]^{1/\tilde{p}} \in L^1 (0,T;L^{\tilde{p}} (\T^d)).
\end{equation}
 To conclude we just need to sum over $j$, consider the $L^p$ norm and use the compressibility estimate on $\Phi$ to change variables. Notice that $\tilde{p} \in (d, p)$ is needed to have $|Du_\delta (s, \cdot)| \in L^{p/\tilde{p}} (\T^d)$ with $p/\tilde{p} >1$ in order to have boundedness of the maximal operator. Notice also that in this case we do not need any compressibility estimate for $\Phi_E$ since the argument of $f$ is $\Phi(t_j)$, whose compressibility is controlled by the hypothesis on the divergence even assuming only $\div (u)^- \in L^1 (0,T; L^\infty (\T^d))$, as we assumed in \eqref{main hypotheses}.

\end{itemize}

 At this point, we can conclude the proof by arguing as in Proposition \ref{Lp stability estimate for regular Lagrangian flows} with $\Delta t$ in place of $\delta$. 
\end{proof}

We are ready to prove Theorem \ref{theorem on estimate between real rlf and euler flow}, whose proof is now trivial:

\begin{proof}[Proof of Theorem \ref{theorem on estimate between real rlf and euler flow}]
    Combine Proposition \ref{Error of regularisation} and Proposition \ref{proposition Lp estimate regularised rlf and euler flow} if $1 < p \leq d$. If $p>d$, we already proved this result in Proposition \ref{proposition Lp estimate regularised rlf and euler flow}.
\end{proof}

\section{Singular probabilistic approximation of $\rho$}

 The goal of this section is to use the results from Section 2 to approximate, in a suitable sense that we will specify shortly, the solution $\rho$ of \eqref{continuity equation}. Let us state a standard result, whose proof we put in the Appendix for the sake of completeness, which will be crucial in order to exploit Theorem \ref{theorem on estimate between real rlf and euler flow} for approximating $\rho$.

\begin{lemma}\label{lemma soluzione cont eqn è push forward via rlf}{\textbf{}}

     Consider a vector field $u$ and an initial density $\rho_0$ satisfying \eqref{main hypotheses}.  Let $\Phi$ be the regular Lagrangian flow of $u$. Then 
      
      $$ \rho(t,\cdot ) \coloneqq  \Phi(t, \cdot )_\# \rho_0 $$ 
      
      is the solution of \eqref{continuity equation}. 
 \end{lemma}

As we mentioned in the introduction the idea is to exploit the Lagrangian representation $\rho (t, \cdot) = \Phi(t, \cdot) _\# \rho_0$ of the solution of \eqref{continuity equation} by constructing a vector field $\Phi_E$ and a measure $\mu$ which approximate $\Phi$ and $\rho_0$ respectively, so that $\rho_E (t) \coloneqq \Phi_E (t, \cdot) _\# \mu$ approximates $\rho$ in some sense. Before stating the main result of this section, let us give the definition of Wasserstein distance. For the sake of simplicity we state this definition and some other results in the metric space $(\mathbb{R}^d , \d)$ rather than $(\T^d, \d)$, but analogous results hold also in the latter case.

\begin{definition}{\textbf{(Wasserstein distance)}}\label{wasserstein distance definition}

Consider $\mu, \nu$ two positive measures in $(\R^d, \d)$, with $\d$ a distance, such that $\mu(\R^d)= \nu(\R^d)$. We define the Wasserstein distance between $\mu $ and $\nu$ as 

\begin{equation}\label{wasserstein distance equation definition}
\W (\mu , \nu) \coloneqq \inf_{\pi \in \Gamma (\mu, \nu)} \int_{\R^d} \int_{\R^d} \operatorname{d}(x,y) d \pi(x,y)  
\end{equation}

where $\Gamma$ is the set of measures on $\R^d \times \R^d$ whose marginals are $\mu $ and $\nu$. If in particular $\d(x,y)= |x-y|$, we denote with $\W_1$ the so called $1$-Wasserstein distance:

\begin{equation} \label{1-Wasserstein distance equation definition}
\W_1 (\mu , \nu) \coloneqq \inf_{\pi \in \Gamma (\mu, \nu)} \int_{\R^d} \int_{\R^d} |x-y| d \pi(x,y) .
\end{equation}

We also define a logarithmic Wasserstein distance with respect to the distance

\begin{equation}\label{distance d_alpha}
\d_{\alpha} (x,y) \coloneqq  \log\left( 1+ \frac{|x-y|}{h^\alpha}\right) \quad \text{for some } \alpha \in [0,1] \text{ and } h>0,
\end{equation}

which we will denote as $\widetilde{\W}_{\alpha}$. We will never explicitly write the dependence of $\d_{\alpha}$ and $\widetilde{\W}_\alpha$ on $h$ in order to have a lighter notation.
\end{definition}

We will not use the above definition, but its dual formulation given by the Kantorovich-Rubinstein duality theorem:

\begin{theorem}{\textbf{(Kantorovich-Rubinstein duality theorem, \cite{villani2021topics})}}

Let $\mu , \nu$ be positive measures with compact support on the metric space $(\R^d, d)$ and with the same mass. Then, it holds that

$$ \inf_{\pi \in \Gamma (\mu, \nu)} \int_{\R^d} \int_{\R^d} \d(x,y) d \pi (x,y) = \sup_{\operatorname{Lip}(f) \leq 1} \int_{\R^d } f(x) d (\mu - \nu )(x),$$

where $f$ is $1$-Lipschitz with respect to the distance $\d$.
\end{theorem}

We can now state the two main results of this section, which we decided to present in two separate Theorems: one for the $1$-Wasserstein distance and the other for the logarithmic one. The reason is that a single statement would have been too long, and while the proofs of both theorems are very similar we believe both of them deserve to be presented in detail. Indeed while on the one hand the $1$-Wasserstein distance is far more commonly used than the logarithmic one, on the other hand the latter is more suited to precisely capture the rate of convergence of the method, as we will see. We recall that the first time such a logarithmic distance was used (at least in the context of the continuity equation in the DiPerna-Lions setting) is in \cite{seis2017quantitative} and that in \cite{schlichting2017convergence,schlichting2018analysis} the authors use such a logarithmic Wasserstein distance to evaluate the rate of convergence of their approximation method, so we decided to use the same distance to have a clearer comparison. We also remark that, for the sake of simplicity, in this section we will restrict ourselves to the case $\rho_0 \in L^\infty (\T^d)$ and non negative. The non negativity assumption is clearly not a problem: we can always split the initial datum as in its positive and negative parts. We will explain how to deal with the unbounded case in Remark \ref{The case rho0 not bounded}. As a final note, we point out that we are forced to distinguish between bounded and unbounded initial densities due to our choice of a uniform probability $\bar{\P}$, which will play a crucial role in Theorem \ref{1 wasserstein theorem} and Theorem \ref{logarithmic wasserstein theorem}. With a different choice of $\bar{\P}$ this issue can be fixed and one does not need to distinguish the two cases. We refer to Remark \ref{remark on unbounded case with ad hoc probability} for a more detailed discussion on the matter. Let us now state the two main results of this section:

\begin{theorem}\textbf{(Singular probabilistic approximation of $\rho$ in $1$-Wasserstein distance)}\label{1 wasserstein theorem}

Assume that $u$ satisfies \eqref{main hypotheses} and $\rho_0 \geq 0$ and bounded, and let $\rho$ be the solution of \eqref{continuity equation}. Assume also that the domain $\T^d$ is partitioned into sets $\{Q_i \}_{i=1} ^N$ such that $\operatorname{diam} (Q_i) \leq \Delta x \; \forall i=1, \dots , N$ for some space discretization scale $\Delta x >0$. Choose, for every set $Q_i$, a random starting point $x_0 ^i \in Q_i$ according to a uniform probability distribution on $Q_i$, and define $M_i \coloneqq \int_{Q_i} \rho_0 dx $. For $t \in [0,T]$ consider the measure given by:

\begin{equation} \label{definition of rho_E}
\rho_E (t,\bar{x}_0) \coloneqq  \sum_{i=1} ^{N} M_i \delta_{\Phi_E (t,x_0 ^i)} = \Phi_E (t, \cdot) _\# \left( \sum_{i=1} ^N M_i\delta_{x_0 ^i}  \right),
\end{equation}

where $\bar{x}_0 = (x_0 ^1 , x_0 ^2 , \dots, x_0 ^N)$ and $\Phi_E$ is defined as in $\eqref{euler flow definition equation}$ with a time step $\Delta t >0$ and with $\delta = \sqrt{\Delta t}$ if $1 < p \leq d$ or as in \eqref{euler flow definition equation p>d} otherwise. Denoting with $\W_1$ the $1$-Wasserstein distance defined in \eqref{1-Wasserstein distance equation definition} it holds that, in expected value, for every $\alpha \in (0,1)$:

\begin{align}\label{expected value desired estimate}
\E[\W_1 (\rho(t,\cdot) , \rho_E (t, \bar{x}_0))] &\lesssim  C_t|\log(\Delta t)| ^{-1} + (\Delta x)^{1- \alpha} +  t \alpha ^{-p} |\log(\Delta x)|^{-p} , 
\end{align}

with $C_t \to 0$ as $t \to 0$. We can also estimate the variance of $\W_1 (\rho(t,\cdot), \rho_E (t, \bar{x}_0))$ if $N \lesssim (\Delta x)^{-d}$. In this case:

\begin{align}\label{variance inequality}
\operatorname{Var}[\W_1 (\rho(t,\cdot), \rho_E (t, \bar{x}_0))] \lesssim C_t ^2 |\log(\Delta t)|^{-2}  + (\Delta x)^{2- 2\alpha} & + t C_t |\log(\Delta t)|^{-1}\\
&+ t^2 \alpha ^{-p}|\log(\Delta x)|^{-p}.\nonumber 
\end{align}

\end{theorem}

\vspace{2em}

\begin{theorem}{\textbf{(Singular probabilistic approximation of $\rho$ in logarithmic Wasserstein distance)}}\label{logarithmic wasserstein theorem}

Consider the same setting of Theorem \ref{1 wasserstein theorem}. It holds that, given the Wasserstein distance $\widetilde{\W}_\alpha$ defined in Definition \ref{wasserstein distance definition} with $h \coloneqq \max\{ \Delta t , \Delta x\}$ and assuming $\alpha= 1/2$, if $1<p \leq d$:

    \begin{equation}\label{bounded expected value}
    \E[\widetilde{\W}_{1/2} (\rho(t, \cdot), \rho_E (t, \bar{x}_0))] \lesssim 1
    \end{equation}

    and if $N \lesssim (\Delta x) ^{-d}$, we also can estimate the variance as

    \begin{equation}\label{logarithmic variance}
    \operatorname{Var}[\widetilde{\W}_{1/2} (\rho(t, \cdot), \rho_E (t, \bar{x}_0))] \lesssim |\log(h)|.
    \end{equation}

    If $p>d$, it holds that for any $\alpha \in (0,1)$:
    
\begin{equation}\label{logarithmic expected value for p>d}
      \E[\widetilde{\W}_{1- \alpha} (\rho(t, \cdot), \rho_E (t, \bar{x}_0))] \lesssim 1 + \frac{|\log(h)|^{1-p}}{\alpha ^p} 
      \end{equation}
    
    and if $N \lesssim (\Delta x) ^{-d}$:
    \begin{equation}\label{logarithmic variance for p>d}
        \operatorname{Var}[\widetilde{\W}_{1- \alpha} (\rho(t, \cdot), \rho_E (t, \bar{x}_0))] \lesssim  |\log(h)| + \frac{|\log(h)|^{2-p}}{\alpha^p} .
    \end{equation}
   
     If moreover $p\geq 2$, we have that:

    \begin{equation}\label{bounded variance}
   \operatorname{Var}[\widetilde{\W}_{1/2} (\rho(t, \cdot), \rho_E (t, \bar{x}_0))] \lesssim 1
    \end{equation}
    if $1 <p \leq d$, and 
    
    \begin{equation}\label{bounded variance p>d}
        \operatorname{Var}[\widetilde{\W}_{1- \alpha} (\rho(t, \cdot), \rho_E (t, \bar{x}_0))] \lesssim 1 + \frac{|\log(h)|^{2-p}}{\alpha ^p} 
    \end{equation} 
    
     if $p>d$.
    
\end{theorem}

Before proving Theorem \ref{1 wasserstein theorem} and Theorem \ref{logarithmic wasserstein theorem}, we will need the following:

\begin{lemma}\label{splitting wasserstein}{\text{}}

    Let $\mu, \nu$ be two positive measures with compact support on $\R^d$ and such that $\mu(\R^d)= \nu(\R^d)$. Let $\mu_i , \nu_i$ with $i= 1, \dots, N$ be positive measures such that

    $$ \mu= \sum_{i=1} ^N \mu_i, \quad  \nu= \sum_{i=1} ^{N} \nu_i \quad \text{and }  M_i=\mu_i (\R^d) = \nu_i (\R^d) \text{ for every } i.$$

    Define also the probability measures
    
    $$ \bar{\mu}_i \coloneqq \mu_i /M_i \text{ and } \bar{\nu}_i \coloneqq \nu_i /M_i.$$

    Then

    \begin{equation}\label{wasserstein inequality}
    \W (\mu, \nu) \leq \sum_{i=1} ^N \W (\mu_i , \nu_i) = \sum_{i=1} ^N M_i \W (\bar{\mu}_i , \bar{\nu}_i) .
    \end{equation}

\end{lemma}
\begin{proof}
    First of all let us prove that for every $i$ it holds 
    
    $$ \W (\mu_i , \nu_i) = M \W_1 (\bar{\mu}_i , \bar{\nu}_i).$$

     By definition:

    $$  \W (\bar{\mu}_i , \bar{\nu}_i )=  \inf_{\pi \in \Gamma (\bar{\mu}_i, \bar{\nu}_i)} \int_{\R^d \times \R^d} \d(x,y) d \pi (x,y) $$

    and given a sequence of measures $\pi_n$ tending to the infimum in \eqref{1-Wasserstein distance equation definition} with marginals $\bar{\mu}_i$ and $\bar{\nu}_i$, $M_i \pi_n$ will tend to the infimum in \eqref{1-Wasserstein distance equation definition} with marginals $\mu_i, \nu_i$, so

    $$ \W_1 (\mu_i , \nu_i) = M_i \W_1 (\bar{\mu}_i, \bar{\nu}_i). $$
    
    For every $i$ and every $\varepsilon>0$ we consider $\pi_i$ a transport plan between $ \bar{\mu}_i$ and $\bar{\nu}_i$ such that

    $$ \int_{\R^d \times \R^d} \d(x,y) d\pi_i (x,y) \leq \W (\bar{\mu}_i , \bar{\nu}_i) + \varepsilon/N .$$

     Then it is easy to see that $\pi^\varepsilon \coloneqq \sum_{i=1} ^N M_i \pi_i \in \Gamma (\mu, \nu)$, since for every $A, B$ Borel:

    $$ \pi^\varepsilon (A \times \R^d) = \sum_{i=1} ^{N} M_i \pi_i (A \times \R^d) = \sum_{i=1} ^N M_i \bar{\mu}_i (A ) =\sum_{i=1} ^N  \mu_i (A) =\mu (A),$$

    and similarly we see that $\pi^\varepsilon (\R^d \times B)= \nu (B)$. Then:

    \begin{align*} 
    &\W (\mu, \nu) \leq \int_{\R^d  \times \R^d} \d(x,y) d \pi ^\varepsilon (x,y) = \sum_{i=1 } ^N M_i \int_{\R^d \times \R^d} \d(x,y) d \pi_i (x,y) \\
    & \leq \sum_{i =1} ^{ N} M_i \W (\bar{\mu}_i , \bar{\nu}_i) + \varepsilon = \sum_{i =1} ^{ N} \W (\mu_i , \nu_i) + \varepsilon.
    \end{align*}

    We conclude since $\varepsilon$ is arbitrarily small.
\end{proof}

\begin{remark}
The above lemma has an heuristic obvious ``proof''. What we are doing is simply to impose that the mass of $\mu_i$ is sent to $\nu_i$ for every $i$, and we minimize this cost for every $i=1 ,\dots, N$. Of course by doing this we are putting more constraints on $\pi \in \Gamma (\mu,\nu)$, so the inequality is just a consequence of the fact that we are considering the infimum in \eqref{wasserstein distance equation definition} on a smaller set than $\Gamma (\mu, \nu)$.
\end{remark}

We are now ready to prove the theorems.

\begin{proof}[Proof of Theorem \ref{1 wasserstein theorem}]

First of all, we notice that in order to prove \eqref{expected value desired estimate} we can work without loss of generality with $\Phi_\delta$ instead of $\Phi$ if $1< p \leq d$. By Proposition \ref{Error of regularisation}, using Kantorovich-Rubinstein duality, we have:

\begin{align}\label{can work wlog with Phi_delta}
     &\W_1 (\Phi(t,\cdot) _\# \rho_0 , \Phi_\delta (t,\cdot) _\# \rho _0 ) \leq \int_{\T^d} \rho_0 (x)|\Phi (t,x) - \Phi_\delta (t,x)| dx  \\
     &\leq \|\rho_0 \|_{L^q} \|\Phi (t, \cdot) - \Phi_\delta (t, \cdot)\|_{L^p} \lesssim C_t |\log(\delta)|^{-1} \approx C_t  |\log(\Delta t )|^{-1}. \nonumber 
\end{align}

At time $t$ fixed, consider $\bar{x}_0  =(x_0 ^1 , \dots, x_0 ^N)$ as a random variable in $((\T^d)^N, \mathcal{B} (\T^d) ^N , \bar{P})$, with $\mathcal{B} (\T^d)$ the Borel sigma algebra on $\T^d$,

\begin{equation}\label{definition of the probability} \bar{\P} = \P_1 \otimes \P_2 \dots \otimes \P_N, \text{ with } \P_i \text{ being a uniform distribution over $Q_i$.}
\end{equation}

We also define $\rho_0 ^i \coloneqq \rho_0 \measurerestr Q_i$ and $\bar{\rho}_0 ^i \coloneqq (\rho_0 \measurerestr Q_i)/M_i$ if $M_i >0$, otherwise $M_i = \bar{\rho}_0 ^i =0$, where $\mu \measurerestr A$ denotes the restriction of a measure $\mu$ to the set $A$. In the following, each time we consider an expected value $\E[ \cdot] $ or a variance $\operatorname{Var}[ \cdot]$ it will be with respect to the probability $\bar{\P}$. We consider

$$ \mu_i \coloneqq  \Phi_\delta (t, \cdot) _\# (\rho_0 \measurerestr Q_i) \text{ and } \nu_i \coloneqq M_i \delta_{\Phi_E (t, x_0 ^i)}= \Phi_E (t, \cdot) _\# (M_i \delta_{x_0 ^i}),$$

 which have the same mass, so we can apply Lemma \ref{splitting wasserstein}. For the moment we will work in a fixed set $Q=Q_i$ of the partition such that $M_i \neq 0$, so in order to keep the notation lighter we will drop the subscript and in the following we will implicitly assume that the measures are restricted to $Q$ and normalized, also we will write $x_0$ instead of $x_0 ^i$. Since $\rho(t,\cdot)  = \Phi_\delta(t,\cdot)_\# \bar{\rho}_0$, we need to estimate $\E [\W_1 (\Phi_\delta (t,\cdot)_\#  \bar{\rho}_0 , \delta_{\Phi_E (t,x_0)})]$. By triangle inequality 

\begin{align}\label{expected value equation}
\E [\W_1 (\Phi_\delta (t,\cdot)_\# \bar{\rho}_0  , \delta_{\Phi_E (t,x_0)})] \leq &\E [\W_1 (\Phi_\delta (t,\cdot)_\# \bar{\rho}_0  , \delta_{\Phi_\delta (t,x_0)})] \\
&+  \E [ \W_1 (\delta_{\Phi_\delta (t,x_0)} , \delta_{\Phi_E (t,x_0)})].\nonumber 
\end{align} 

For the first addendum in \eqref{expected value equation}, recalling that $\int_Q \bar{\rho}_0 (x) dx=1$:

\begin{align}\label{1-wasserstein estimate}
&\W_1 (\Phi_\delta (t,\cdot) _\# \bar{\rho}_0  , \delta_{\Phi_\delta(t,x_0)}) \\
&=\sup_{ \operatorname{Lip}(h) \leq 1}  \left(\int_Q \bar{\rho}_0 (x) h (\Phi_\delta(t,x)) dx - \int_Q \bar{\rho}_0 (x) h (\Phi_\delta(t,x_0)) dx \right) \nonumber  \\
&\leq \int_Q \bar{\rho}_0 (x)  |\Phi_\delta(t,x)- \Phi_\delta(t,x_0)| dx. \nonumber
\end{align}

Using Lemma \ref{Lishitz rlf} we can consider a set $K \subset \T^d$ with $|K^c|=  \alpha^{-p} c_d ^p A_p (R,\Phi_\delta) ^p |\log(\Delta x)|^{-p} $ for $\alpha \in (0,1)$ such that $\operatorname{Lip} (\Phi_\delta(t, \cdot) \measurerestr K) \leq (\Delta x)^{- \alpha}$ for every $0 \leq t \leq T$. Notice that $A_p (R, \Phi_\delta)$ does not depend on the mollification scale $\delta$, since $\|D_x u_\delta \|_{L^1 (L^p)} \leq \|D_x u\|_{L^1 (L^p)}$ for every $\delta >0$ by Young's inequality. Then using that $|x- x_0 | \leq \operatorname{diam} (Q) \leq \Delta x$ and \eqref{1-wasserstein estimate}:

\begingroup
\allowdisplaybreaks
\begin{align}\label{1 wasserstein estimate first addendum}
    &\E\left[\W_1 (\Phi_\delta(t,\cdot)_\# \bar{\rho}_0 , \delta_{\Phi_\delta(t,x_0)
    })\right] \\
    &\leq \frac{1}{\Le^d (Q)}\left[ \int_{Q \cap K} \left( \int_{Q \cap K} \bar{\rho}_0 (x) |\Phi_\delta(t,x)- \Phi_\delta(t,x_0)| dx\right) dx_0  \right.\nonumber \\
    & + \int_{Q \cap K} \left(\int_{Q \cap K^c} \bar{\rho}_0 (x) |\Phi_\delta(t,x)- \Phi_\delta(t,x_0)| dx \right) dx_0  \nonumber   \\
    &+ \left. \int_{Q \cap K ^c}  \left(\int_{Q} \bar{\rho}_0 (x) |\Phi_\delta(t,x)- \Phi_\delta(t,x_0)| dx \right) dx_0 \right] \nonumber  \\
    & \lesssim   (\Delta x)^{1- \alpha} \|\bar{\rho}_0\|_{L^1 (Q \cap K)}  +  (\Delta x+ 2t\|u\|_\infty) \;  \left(\|\bar{\rho}_0 \|_{L^1 (Q \cap K^c)} + \frac{\Le^d (Q \cap K^c)}{\Le^d (Q)} \right), \nonumber 
\end{align}
\endgroup

where we used that 

$$ |\Phi_\delta(t, x) - \Phi_E (t, x_0)| \leq \Delta x + 2t\|u\|_{\infty} \text{ and that } \|\bar{\rho}_0 \|_{L^1 (Q)} =1 \text{ by construction.}$$

 As for the other addendum in \eqref{expected value equation}, we have:

\begin{align} \label{Wasserstein distance between dirac deltas}
\W_1 (\delta_{\Phi_\delta (t,x_0)}, \delta_{\Phi_E (t,x_0)})&=  \sup_{\operatorname{Lip}(h)=1}  \left(\int_{\T^d} h (x) d \delta_{\Phi_\delta (t,x_0)} - \int_{\T^d} h(x) d \delta_{\Phi_E (t,x_0)} \right)\\
&\leq  |\Phi_\delta(t,x_0) - \Phi_E (t,x_0)|. \nonumber
\end{align}

So, considering the expected value:

\begin{align} \label{expected value second addendum}
&\E [\W_1 (\delta_{\Phi_\delta (t,x_0)}, \delta_{\Phi_E (t,x_0)})] \leq  \frac{1}{\Le ^d (Q)}\int_Q  |\Phi_\delta(t,x_0) - \Phi_E (t,x_0)| dx_0.
\end{align}

Then using Lemma \ref{splitting wasserstein}:

\begingroup
\allowdisplaybreaks
\begin{align}\label{expected value inequality} 
&\E[\W_1 (\Phi_\delta (t, \cdot) _\# \rho_0 , \rho_E (t,\bar{x}_0))] \leq \E \left[ \sum_{i=1} ^{N} M_i \W_1 \left(\Phi_\delta (t,\cdot)_\# \bar{\rho}_0 ^i , \delta_{\Phi_E (t,x)} \right)\right] \\
&=\sum_{i=1} ^{N} M_i \E \left[\W_1 \left(\Phi_\delta (t,\cdot)_\# \bar{\rho}_0 ^i, \delta_{\Phi_E (t,x)} \right)\right]  \nonumber  \\
&\lesssim \sum_{i=1} ^N  M_i 
\begin{aligned}[t]
& \left( \vphantom{\int} (\Delta x)^{1- \alpha} \|\bar{\rho}_0 ^i\|_{L^1 (Q_i \cap K)}  \right. + (\Delta x + 2t\|u\|_\infty) \|\bar{\rho}_0 ^i \|_{L^1 (Q_i \cap K^c)} \\
&+(\Delta x + 2 t \|u\|_\infty ) \frac{\Le^d (Q \cap K^c)}{\Le^d (Q)} + \left. \frac{1}{\Le^d (Q_i)}\int_{Q_i} |\Phi_E (t,x) - \Phi_\delta(t,x)|dx \right)
\end{aligned} 
\nonumber \\
&= \sum_{i=1} ^N   
\begin{aligned}[t] 
&\left( \vphantom{\int} (\Delta x)^{1- \alpha} \|\rho_0 ^i \|_{L^1 (Q_i \cap K)} \right.+ (\Delta x + 2t\|u\|_\infty) \|\rho_0 ^i \|_{L^1 (Q_i \cap K^c)} \\
&+ \|\rho_0 \|_\infty (\Delta x + 2t\|u\|_\infty) \Le ^d (Q_i \cap K^c) \left. + \|\rho_0\|_{\infty}  \int_{Q_i} |\Phi_E (t,x) - \Phi_\delta(t,x)|dx \right)
\end{aligned}\nonumber \\
& \lesssim (\Delta x)^{1- \alpha} \|\rho_0\|_{L^1 (\T^d)}  + t \|\rho_0\|_{L^1 (K^c)} + t \|\rho_0\|_\infty \Le ^d (K^c) + \|\rho_0\|_{\infty} \|\Phi_\delta(t,x) - \Phi_E (t,x)\|_{L^1 (\T^d)}  \nonumber \\
&\lesssim (\Delta x)^{1- \alpha}  + \|\rho_0\|_{\infty} \left(t  \frac{|\log(\Delta x)|^{-p}}{\alpha ^p}  +  C_t |\log(\Delta t)|^{-1}  \right), \nonumber 
\end{align}
\endgroup

where we use that $\|\rho_0 \|_{L^1 (K^c)} \leq \|\rho_0\|_{\infty} \Le^d (K^c) $ with $\Le^d (K^c) \approx \alpha ^{-p} |\log(\Delta x)|^{-p}$. We also have

\begin{equation}\label{Lp final estimate needed for rlf}\|\Phi_\delta(t,x) - \Phi_E (t,x)\|_{L^1 (\T^d)} \lesssim C_t |\log(\Delta t)|^{-1}
\end{equation}

by H\"older inequality and Proposition \ref{proposition Lp estimate regularised rlf and euler flow}, and thus we have proved \eqref{expected value desired estimate}. We now estimate the variance. By definition

$$ \operatorname{Var}[X] = \E[X^2] - \E[X] ^2 .$$

We will simply estimate $\operatorname{Var}(X)$ from above with $\E[X^2]$. First of all we verify that, in order to prove \eqref{variance inequality}, if $1 < p \leq d$ we can work without loss of generality with $\Phi_\delta$ instead of $\Phi$. Indeed:

\begin{equation}\label{can use Phi_delta instead of Phi}
    \W_1 (\Phi (t,\cdot)_\# \bar{\rho}_0  , \rho_E (t, \bar{x}_0))^2 \leq 2 \W_1 (\Phi (t,\cdot)_\# \bar{\rho}_0  , \Phi_\delta (t,\cdot)_\# \bar{\rho}_0 )^2 + 2\W_1 (\Phi_\delta  (t,\cdot)_\# \bar{\rho}_0  , \rho_E (t, \bar{x}_0))^2,
\end{equation}

and by \eqref{can work wlog with Phi_delta} we have 

\begin{equation}\label{can work with Phi_delta for the variance}
\W_1 (\Phi (t,\cdot)_\# \bar{\rho}_0  , \Phi_\delta (t,\cdot)_\# \bar{\rho}_0 )^2 \lesssim C_t ^2 |\log(\Delta t)|^{-2}.
\end{equation}

We now restrict to working in a fixed set $Q_i$. As we did before, we omit writing the index $i$ and we will denote $(\rho_0 \measurerestr Q_i )/M_i$ and $x_0 ^i$ as $\bar{\rho}_0$ and $x_0$. It holds:

\begin{align}\label{(a+b)^2 <}
\W_1 (\Phi_\delta (t,\cdot)_\# \bar{\rho}_0  , \delta_{\Phi_E (t,x_0)})^2  \leq 2 \W_1 (\Phi_\delta (t,\cdot)_\# \bar{\rho}_0  ,\delta_{\Phi_\delta (t,x_0)})^2 + 2 \W_1 (\delta_{\Phi_\delta (t,x_0)} , \delta_{\Phi_E (t,x_0)})^2. 
\end{align}

For the first addendum in the right hand side of \eqref{(a+b)^2 <}, using \eqref{1-wasserstein estimate}, the same choice of $K$, the decomposition in \eqref{1 wasserstein estimate first addendum} and the elementary inequality $(a+ b + c) ^2 \leq 4a^2 + 4b^2 + 4 c^2$:

\begin{align}\label{expected value 1-wasserstein estimate squared}
     \E[\W_1 (\Phi_\delta(t,\cdot)_\# \bar{\rho}_0 , \delta_{\Phi_\delta(t,x_0)
    })^2]  \lesssim (\Delta x)^{2- 2 \alpha} \|\bar{\rho}_0 \|_{L^1 (Q \cap K)} ^2 &+ (\Delta x + 2 \|u\|_\infty t)^2 \|\bar{\rho}_0\|_{L^1 (Q \cap K^c)} ^2  \\
    &+ (\Delta x + 2 \|u\|_\infty t)^2 \frac{\Le ^d (Q \cap K^c) ^2}{\Le (Q)^2}. \nonumber
\end{align}

For the second addendum in \eqref{(a+b)^2 <}, using \eqref{Wasserstein distance between dirac deltas} in expected value:

\begin{align} \label{devo usarti sotto e non so come chiamarti}
&\E [\W_1 (\delta_{\Phi_\delta (t,x_0)}, \delta_{\Phi_E (t,x_0)})^2] \leq \frac{1}{\Le ^d (Q)}\int_Q|\Phi_\delta(t,x_0) - \Phi_E (t,x_0)|^2 dx_0 \\
&\leq  \frac{\|\Phi_\delta (t, \cdot)- \Phi_E (t,\cdot) \|_\infty}{\Le^d (Q)} \int_Q  |\Phi_\delta (t, x_0)- \Phi_E (t, x_0) | dx_0 \nonumber \\  
&\leq \frac{2\|u\|_\infty t}{\Le^d (Q)} \int_Q |\Phi_\delta (t, x_0)- \Phi_E (t, x_0) | dx_0 .\nonumber 
\end{align}

The conclusion follows as in \eqref{wasserstein inequality}, using \eqref{expected value 1-wasserstein estimate squared}, \eqref{devo usarti sotto e non so come chiamarti} and that by Cauchy-Schwartz: 

$$ \W_1 (\mu, \nu)  \leq \sum_{i=1} ^N M_i  \W_1 (\bar{\mu}_i, \bar{\nu}_i)  \implies \W_1 (\mu, \nu)^2 \leq N \sum_{i=1} ^N M_i ^2 \W_1 (\bar{\mu}_i , \bar{\nu}_i )^2 .  $$

Indeed 

\allowdisplaybreaks

\begin{align}\label{variance sum estimate}
    &\E[\W_1 (\Phi (t, \cdot) _\# \rho_0 , \rho_E (t,\bar{x}_0)) ^2] \leq  N \sum_{i=1} ^{N} M_i ^2 \E \left[\W_1 \left(\Phi (t,\cdot)_\# \bar{\rho}_0 ^i, \delta_{\Phi_E (t,x)} \right) ^2\right]  \\
    &\lesssim N \sum_{i=1} ^N  M_i ^2 
    \begin{aligned}[t] 
    &\left( \vphantom{\int} (\Delta x)^{2- 2 \alpha} \|\bar{\rho}_0 ^i \|_{L^1 (Q_i \cap K)} ^2 \right. +  (\Delta x+t)^2\|\bar{\rho}_0 ^i \|_{L^1 (Q_i \cap K^c)}^2 \\
    &+(\Delta x+t)^2 \frac{\Le ^d (Q_i \cap K^c)^2}{\Le ^d (Q_i)^2 } + \left. \frac{t}{\Le^d(Q_i)} \int_{Q_i}  |\Phi_E (t,x) - \Phi_\delta(t,x)| dx \right)
    \end{aligned} \nonumber \\
    & \leq   N \sum_{i=1} ^N   M_i 
    \begin{aligned}[t]
    & \left( \vphantom{\int} (\Delta x)^{2- 2\alpha} \|\rho_0 ^i\|_{L^1 (Q_i )} \right. + (2(\Delta x)^2+2t^2) \|\rho_0 ^i \|_{L^1 (Q_i \cap K^c)} \\
&+(2(\Delta x)^2 + 2t^2)\|\rho_0\|_\infty \Le(K^c \cap Q_i) +\left. t \|\rho_0\|_{\infty}\int_{Q_i} |\Phi_E (t,x) - \Phi_\delta(t,x)| dx \right).
    \end{aligned}\nonumber \\
    &\leq   N \sum_{i=1} ^N   M_i 
    \begin{aligned}[t]
    & \left( \vphantom{\int} (\Delta x)^{2- 2\alpha} \|\rho_0 ^i\|_{L^1 (Q_i )} \right. + (4(\Delta x)^2+4t^2) \|\rho_0\|_{\infty} \Le(K^c \cap Q_i) \\
&+\left. t \|\rho_0\|_{\infty} \int_{Q_i} |\Phi_E (t,x) - \Phi_\delta(t,x)| dx \right).
    \end{aligned}\nonumber 
\end{align}

By the assumption on $\operatorname{diam}(Q_i)$ we have $M_i \lesssim  \|\rho_0 \|_\infty (\Delta x)^d$ for every $i$, and also by assumption $N \lesssim (\Delta x)^{-d}$ , so 

\begin{align*} 
&\E[\W_1 (\Phi_\delta (t, \cdot) _\# \rho_0 , \rho_E (t,\bar{x}_0)) ^2] \\
&\begin{aligned}
\lesssim \|\rho_0\|_{\infty} \sum_{i=1} ^N \left( \vphantom{\int} (\Delta x)^{2- 2 \alpha} \|\rho_0 ^i\|_{L^1 (Q_i )} + t^2 \|\rho_0\|_{\infty} \Le ^d (K^c \cap Q_i) \right.\\
\left. + t \| \rho_0 \|_{L^\infty} \int_{Q_i}|\Phi_E (t,x) - \Phi_\delta(t,x)| dx \right)
\end{aligned}\\
&\begin{aligned}
    = \|\rho_0\|_{\infty} \left(\vphantom{\int} (\Delta x)^{2- 2 \alpha} \|\rho_0 \|_{L^1 (\T^d) } + t^2 \|\rho_0\|_{\infty} \Le^d (K^c)\right. \\
    \left.+ t \|\rho_0\|_{\infty}  \int_{\T^d}|\Phi_E (t,x) - \Phi_\delta(t,x)| dx \right)
    \end{aligned}\\
&\begin{aligned}
    \leq \|\rho_0\|_{\infty} \left( \vphantom{\int}(\Delta x)^{2- 2 \alpha} \|\rho_0 \|_{L^1 (\T^d) } + t^2 \|\rho_0\|_{\infty} \Le^d (K^c)\right. \\
    \left.+ t \|\rho_0\|_{\infty}\int_{\T^d} |\Phi_E (t,x) - \Phi_\delta(t,x)| dx\right) 
    \end{aligned}\\
&\lesssim \|\rho_0\|_{\infty} \left( (\Delta x)^{2- 2 \alpha} \|\rho_0 \|_{L^1 (\T^d) } + t^2 \|\rho_0\|_{\infty} \alpha ^{-p} |\log(\Delta x)|^{-p} + t C_t  \|\rho_0\|_{\infty} |\log(\Delta t)|^{-1} \right),
\end{align*}

and together with \eqref{can use Phi_delta instead of Phi} and \eqref{can work with Phi_delta for the variance} we can conclude and get \eqref{variance inequality}. 
\end{proof}

\begin{remark}
    A simple sufficient condition to impose in order to have that $N \lesssim (\Delta x)^{-d}$ is a lower bound on the volume of the cells, i.e. $|Q_i| \geq c (\Delta x)^d$ for some $c>0$ for every $i$. 
    \end{remark}

\begin{remark}[The role of $\alpha$]\label{remark explanation of alpha}  
    We will try to give an heuristic explanation of the role of $\alpha$. It essentially takes into account the fact that one cannot estimate $\rho(t, \cdot)$ better than a certain threshold even for small times, since we have a discretization of the initial datum at scale $\Delta x$. We can check that if $t=0$, we can consider $\alpha =1$ and estimate $\W_1 \left(\rho_0 , \sum_{i=1} ^N \delta_{x_0 ^i} \right) $ deterministically. Indeed, let $\bar{\rho}_0 ^i$ be the probability distribution in $Q_i$ defined in the statement of Theorem \ref{1 wasserstein theorem}. By Kantorovich-Rubinstein:
    
    \begin{align*}
    &\W_1 (\bar{\rho}_0 ^i , \delta_{x_0 ^i} )= \sup_{\operatorname{Lip} (f) \leq 1} \left(\int_{Q_i} f(x) \bar{\rho}_0 ^i  (x) dx - \int_{Q_i} f(x_0) \bar{\rho}_0 ^i (x) dx \right) \\
    &= \int_{Q_i} \bar{\rho}_0 ^i (x) |f(x)- f(x_0)| dx \leq \Delta x \int_{Q_i} \bar{\rho}_0 ^i (x) dx ,
    \end{align*}
    
    and we can easily conclude applying Lemma \ref{splitting wasserstein}.
\end{remark}

\begin{remark}[The role of the expected value]\label{remark that maybe estimates hold pointwise}
    If we did not consider any expected value in Theorem \ref{1 wasserstein theorem}, instead of \eqref{expected value second addendum} we would have had to estimate pointwise differences between $\Phi_\delta$ and $\Phi_E$, which we are unable to control. Indeed, while it is certainly possible to control in a pointwise sense the difference if we allow either $\Delta t$ or $\Delta x$ to be much smaller than the other, in the more interesting case where $\Delta t = \Delta x=h$ for some $h>0$, it is not clear if a pointwise estimate holds as $h \to 0$.
\end{remark}

Before proving Theorem \ref{logarithmic wasserstein theorem}, let us remark once again why it is more precise than Theorem \ref{1 wasserstein theorem}. While before we had for every $\alpha \in (0,1)$ essentially the same logarithmic rate, in this case defining $h \coloneqq \max\{\Delta t, \Delta x\}$ the optimal rate $h^\alpha$ for which we will be able to estimate both the expected value and the variance will be $1/2$ for $1< p \leq d$, while for $p>d$ we get better estimates. Namely, if $p>d$ we can estimate the $\widetilde{\W}_{1-\alpha}$-distance, but with constants that diverge to $+ \infty$ as $\alpha \to 0$. We thus do not have a sharp rate for this case. We notice that the convergence rate of $1/2$ is likely sharp if $1<p \leq d$ (as shown in \cite{schlichting2017convergence} for the upwind scheme).

\begin{proof}[Proof of Theorem \ref{logarithmic wasserstein theorem}]

    Since the relevant quantity will be $h$ in the following, we can assume without loss of generality that $h = \Delta t = \Delta x$. If $1< p \leq d$, we first need to estimate $ \widetilde{\W}_{1/2} (\Phi(t, \cdot)_\# \rho_0 , \Phi_\delta (t, \cdot)_\# \rho_0)$ since we want to work with $\Phi_\delta$ in lieu of $\Phi$ with $\delta = \sqrt{h}$. Notice that this will not be necessary if $p >d$. By Kantorovich-Rubinstein duality

    \begin{align}\label{wasserstein distance phi , phi_delta}
    \widetilde{\W}_{1/2} (\Phi(t, \cdot)_\# \rho_0 , \Phi_\delta (t, \cdot)_\# \rho_0) &\leq \int_{\T^d} \rho_0 (x) \log \left(1+ \frac{|\Phi (t,x) - \Phi_\delta (t,x)|}{\sqrt{h}}  \right) dx\\
    &\leq \|\rho_0 \|_{L^q (\T^d)} \left| \left| \log \left( 1 + \frac{|\Phi(t, \cdot)- \Phi_\delta (t, \cdot)|}{\sqrt{h}}\right)\right| \right|_{L^p (\T^d)} \leq C ,\nonumber 
    \end{align} 

where we have used the estimate of Remark \ref{remark on boundedness of logarithm Lp estimate with Lp norm of difference velocities}, since $\|u- u_\delta \|_{L^1 (L^p)} \lesssim \sqrt{h}$, because $\delta = \sqrt{h}$. Arguing as in Theorem \ref{1 wasserstein theorem} we restrict to working in a fixed set $Q_i$ and we will omit the index for the moment to lighten the notation, in particular $\bar{\rho}_0$, $x_0$ and $Q$  will denote $(\rho_0 \measurerestr Q_i) /M_i$, $x_0 ^i$ and $Q_i$ respectively. By triangle inequality:

\begin{align}\label{second theorem triangle inequality}
\E \left[\widetilde{\W}_{1/2} (\Phi_\delta (t,\cdot)_\# \bar{\rho}_0  , \delta_{\Phi_E (t,x_0)})\right] &\leq \E \left[\widetilde{\W}_{1/2} (\Phi_\delta (t,\cdot)_\# \bar{\rho}_0  , \delta_{\Phi_\delta (t,x_0)})\right] \\
&+  \E \left[ \widetilde{\W}_{1/2} (\delta_{\Phi_\delta (t,x_0)} , \delta_{\Phi_E (t,x_0)})\right].\nonumber 
\end{align} 

For the first addendum in the right hand side of \eqref{second theorem triangle inequality} it holds the analogous of \eqref{1-wasserstein estimate}, i.e. 

\begin{equation*} \widetilde{\W}_{1/2} (\Phi_\delta (t,\cdot)_\# \bar{\rho}_0  , \delta_{\Phi_\delta (t,x_0)}) \leq  \int_Q \bar{\rho}_0 (x)  \log \left(1 + \frac{|\Phi_\delta(t,x)- \Phi_\delta(t,x_0)|}{\sqrt{h}} \right) dx. \end{equation*}

By Lemma \ref{Lishitz rlf} we can choose $K \subset \T^d$ such that $|\T^d \setminus K|= 2^p c_d ^p A_p (R,\Phi_\delta)^p |\log(h)|^{-p} $ and $\operatorname{Lip} (\Phi_\delta (t, \cdot) \measurerestr K) \leq 1/\sqrt{h}$ for every $t$. Notice that $A_p (R, \Phi_\delta)$ does not depend on the mollification scale $\delta$, since $\|D_x u_\delta \|_{L^1 (L^p)} \leq \|D_x u\|_{L^1 (L^p)}$ for every $\delta >0$. Then it holds

\allowdisplaybreaks
\begin{align*}
    &\E[ \widetilde{\W}_{1/2} (\Phi_\delta (t,\cdot)_\# \bar{\rho}_0 , \delta_{\Phi_\delta (t,x_0)
    }) ] \\
    &\leq \frac{1}{\Le^d (Q)} \left[ \int_{Q \cap K}  \left( \int_{Q \cap K} \bar{\rho}_0 (x)\log \left( 1+ \frac{|\Phi_\delta(t,x)- \Phi_\delta(t,x_0)|}{\sqrt{h}} \right)dx\right) dx_0  \right.\\
    & + \int_{Q \cap K}  \left(\int_{Q \cap K^c} \bar{\rho}_0 (x)\log \left( 1+ \frac{|\Phi_\delta(t,x)- \Phi_\delta(t,x_0)|}{\sqrt{h}} \right)dx \right) dx_0   \\
    &+ \left.\int_{Q \cap K ^c}  \left(\int_{Q} \bar{\rho}_0 (x) \log \left( 1+ \frac{|\Phi_\delta(t,x)- \Phi_\delta(t,x_0)|}{\sqrt{h}} \right)dx \right) dx_0 \right] \\
    & \lesssim   \|\bar{\rho}_0\|_{L^1 (Q\cap K)} +  |\log(h)| \;  \|\bar{\rho}_0 \|_{L^1 (Q \cap K^c)} +|\log(h)| \frac{\Le ^d (Q \cap K^c)}{\Le^d (Q)},
\end{align*}

where in the first addendum we used the Lipschitz estimate on $\Phi_\delta (t, \cdot)\measurerestr K$, while for the second and third addenda we simply estimated the logarithm as $|\log(h)|$ and used that $\|\bar{\rho}_0\|_{L^1 (Q)} =1$. For the second addendum in \eqref{second theorem triangle inequality}, using the analogous of \eqref{Wasserstein distance between dirac deltas}:

\begin{align*} 
\E [\widetilde{\W}_{1/2} (\delta_{\Phi (t,x_0)}, \delta_{\Phi_E (t,x_0)})] &\leq \frac{1}{\Le ^d (Q)}\int_Q  \log \left( 1+ \frac{|\Phi_\delta (t,x_0) - \Phi_E (t,x_0)|}{\sqrt{h}}\right) dx_0.
\end{align*}

So, estimating as in \eqref{expected value inequality}:

\begin{align}\label{log Wasserstein distance phi delta rho e}
&\E[\widetilde{\W}_{1/2} (\Phi_\delta (t, \cdot) _\# \rho_0 , \rho_E (t,\bar{x}_0))] \leq \E \left[ \sum_{i=1} ^{N} M_i \widetilde{\W}_{1/2} \left(\Phi_\delta (t,\cdot)_\# \bar{\rho}_0 ^i , \delta_{\Phi_E (t,x)} \right)\right] \\
&=\sum_{i=1} ^{N} M_i \E \left[\widetilde{\W}_{1/2} \left(\Phi_\delta (t,\cdot)_\# \bar{\rho}_0 ^i, \delta_{\Phi_E (t,x)} \right)\right]  \nonumber  \\
&\lesssim \sum_{i=1} ^N  M_i
\begin{aligned}[t]
& \left( \vphantom{\log \left( 1+ \frac{|\Phi(t,x) - \Phi_E (t,x)|}{\sqrt{h}}\right)}\|\bar{\rho}_0 ^i\|_{L^1 (Q_i \cap K)}  \right. + |\log(h)|\;  \|\bar{\rho}_0 ^i \|_{L^1 (Q_i \cap K^c)} \\
& \left.+|\log(h)| \frac{\Le ^d (Q_i \cap K)}{\Le ^d (Q)}+  \frac{1}{\Le^d (Q_i)}\int_{Q_i}\log \left( 1+ \frac{|\Phi_\delta (t,x) - \Phi_E (t,x)|}{\sqrt{h}}\right)dx \right) 
\end{aligned}
\nonumber \\
&\leq \sum_{i=1} ^N   
\begin{aligned}[t] 
&\left( \vphantom{\log \left( 1+ \frac{|\Phi(t,x) - \Phi_E (t,x)|}{\sqrt{h}}\right)}\|\rho_0 ^i \|_{L^1 (Q_i \cap K)} \right. +|\log(h)| \; \|\rho_0\|_{\infty} \Le^d (Q_i \cap K^c)\\
&+ |\log(h)| \; \|\rho_0\|_{\infty} \Le ^d (Q_i \cap K^c) + \|\rho_0\|_{\infty} \left.\int_{Q_i}  \log \left( 1+ \frac{|\Phi_\delta (t,x) - \Phi_E (t,x)|}{\sqrt{h}}\right)dx \right) 
\end{aligned}\nonumber \\
&\lesssim    \|\rho_0\|_{L^1 (\T^d)} + |\log(h)| \;\|\rho_0\|_{\infty} \;\Le^d (K^c) + \|\rho_0\|_{\infty}  \int_{\T^d}  \log \left( 1+ \frac{|\Phi_\delta (t,x) - \Phi_E (t,x)|}{\sqrt{h}}\right)dx \nonumber \\
& \lesssim \|\rho_0\|_{L^1 (\T^d)} + \|\rho_0\|_{\infty} \left( |\log(h)|^{1-p} + 1 \right)  \leq C.\nonumber
\end{align}

Combining \eqref{wasserstein distance phi , phi_delta} and \eqref{log Wasserstein distance phi delta rho e} we get \eqref{bounded expected value}. With the same estimates used in Theorem \ref{1 wasserstein theorem}, we can also recover the estimate for the variance. First of all, let us see that if $1 < p \leq d$ we can work without loss of generality with $\Phi_\delta$ in lieu of $\Phi$. Indeed we can use \eqref{can use Phi_delta instead of Phi} and \eqref{wasserstein distance phi , phi_delta} with $\widetilde{\W}_{1/2}$  to get $\widetilde{\W}_{1/2} (\Phi(t,\cdot)_\# \bar{\rho}_0 ,\Phi_\delta(t,\cdot)_\# \bar{\rho}_0) ^2 \leq C$. We are left to estimate $\widetilde{\W}_{1/2} (\Phi_\delta (t, \cdot)_\# \rho_0 , \rho_E (t, \bar{x}_0))$ and by \eqref{(a+b)^2 <} we first estimate: 

\allowdisplaybreaks
\begin{align*}
    &\E[ \widetilde{\W}_{1/2} (\Phi_\delta (t,\cdot)_\# \bar{\rho}_0 , \delta_{\Phi_\delta (t,x_0)
    }) ^2] \\
    &\lesssim \frac{1}{\Le^d (Q)} \left[ \int_{Q \cap K}  \left( \int_{Q \cap K} \bar{\rho}_0 (x)\log \left( 1+ \frac{|\Phi_\delta(t,x)- \Phi_\delta(t,x_0)|}{\sqrt{h}} \right)dx\right)^2 dx_0  \right.\\
    & + \int_{Q \cap K}  \left(\int_{Q \cap K^c} \bar{\rho}_0 (x)\log \left( 1+ \frac{|\Phi_\delta(t,x)- \Phi_\delta(t,x_0)|}{\sqrt{h}} \right)dx \right)^2 dx_0   \\
    &+ \left.\int_{Q \cap K ^c}  \left(\int_{Q} \bar{\rho}_0 (x) \log \left( 1+ \frac{|\Phi_\delta(t,x)- \Phi_\delta(t,x_0)|}{\sqrt{h}} \right)dx \right)^2 dx_0 \right] \\
    & \lesssim   \|\bar{\rho}_0\|_{L^1 (Q\cap K)} +  |\log(h)|^2 \;  \|\bar{\rho}_0 \|_{L^1 (Q \cap K^c)} + |\log(h)|^2 \frac{\Le ^d (Q \cap K^c)}{\Le ^d (Q)},
\end{align*}

where in the first addendum we used the Lipschitz estimate on $\Phi_\delta (t, \cdot) \measurerestr K$ and in the last two addenda we simply estimated the logarithm as $|\log(h)|$. The other relevant estimate is:

\begin{align} \label{estimate square of log}
\E\left[\widetilde{\W}_{1/2} (\delta_{\Phi_\delta (t,x_0)} , \delta_{\Phi_E (t, x_0)} ) ^2 \right] &\leq \frac{1}{\Le^d(Q)} \int_Q   \log \left( 1+ \frac{|\Phi_\delta (t,x_0) - \Phi_E (t,x_0)|}{\sqrt{h}} \right)^2 dx_0\\
&\lesssim \frac{|\log(h)|}{\Le^d (Q)} \int_Q   \log \left( 1+ \frac{|\Phi_\delta (t,x_0) - \Phi_E (t,x_0)|}{\sqrt{h}} \right) dx_0.\nonumber 
\end{align}

Summing over $j$ as in \eqref{variance sum estimate} we have  

\begin{align*} 
& \operatorname{Var}[\widetilde{\W}_{1/2} (\rho(t, \cdot), \rho_E (t, \bar{x}_0))]\lesssim \|\rho_0\|_{\infty} \left[ \vphantom{\int \log \left( 1+ \frac{|\Phi_\delta (t,x_0) - \Phi_E (t,x_0) |}{\sqrt{h}} \right)}\|\rho_0\|_{L^1 (\T^d)} + |\log(h)|^2 \|\rho_0\|_{\infty} \Le ^d ( K^c) \right.\\
&\left.+ |\log(h)| \; \|\rho_0\|_{\infty}\int_{\T^d}  \log \left( 1+ \frac{|\Phi_\delta (t,x_0) - \Phi_E (t,x_0) |}{\sqrt{h}} \right) dx_0 \right]\nonumber  \\
&\lesssim \|\rho_0\|_{\infty} (\|\rho_0\|_{L^1 (\T^d)} + \|\rho_0\|_{\infty} \; |\log(h)|^{2-p} + \|\rho_0\|_{\infty} \; |\log(h)| )\\
&\lesssim |\log(h)|,
\end{align*}

so we have proved \eqref{logarithmic variance}.
The stronger rate for $p >d$ can be proved by noticing that we do not need to first estimate \eqref{wasserstein distance phi , phi_delta} and by defining $K$ such that 
$$|\T^d \setminus K| \leq \alpha^{-p} c_d ^p A_p (R, \Phi) ^p |\log(h)|^{-p} \text{ and } \operatorname{Lip}(\Phi(t,\cdot) \measurerestr K) \leq h^{- \alpha} \text{ for every }t. $$ 

Every estimate works the same way, but keeping track of the dependence on $\alpha$ gives us \eqref{logarithmic expected value for p>d} and \eqref{logarithmic variance for p>d}. Finally, if we also assume $p \geq 2$ we can keep the square in the integrand on the first line of \eqref{estimate square of log}, and with the same estimates as above we can conclude:

\begin{align*}
    \operatorname{Var}[\widetilde{\W}_{1/2} (\rho(t, \cdot), \rho_E (t, \bar{x}_0))]&\lesssim \|\rho_0\|_{L^1 (\T^d)} + \|\rho_0\|_{L^\infty} |\log(h)|^{2-p} \\
    &+ \|\rho_0\|_{\infty} \int_{\T^d}  \log \left( 1+ \frac{|\Phi_\delta (t,x_0) - \Phi_E (t,x_0)|}{\sqrt{h}} \right) ^2 dx_0.
\end{align*}

Since $p\geq 2$ we can use estimate the last integral using H\"older inequality and \eqref{Lp norm of logarithm}, and we have proved \eqref{bounded variance}. The estimate \eqref{bounded variance p>d} can be proved similarly, keeping track of the dependence of $\alpha$ on $\Le ^d (K^c)$.
\end{proof}

\begin{remark}[The case $\rho_0 \notin L^\infty (\T^d)$]\label{The case rho0 not bounded}
As already said, for the sake of simplicity we assumed $\rho_0$ to be bounded. In the general case, one can simply estimate $\rho_0$ with $\rho_0 \mathbb{1}_{\{ \rho_0 \leq K\}}$, with $\mathbb{1}_A$ denoting the indicator function of the set $A$. After this first approximation one should keep track of the dependence of the estimates on $K$.
\end{remark}

\begin{remark}[A possible alternative choice for $\bar{P}$]\label{remark on unbounded case with ad hoc probability}

   If the choice of the point $x_0 ^i$ in the cell $Q_i$ is not made according to a uniform distribution, but according to the probability density function $(\rho_0 \measurerestr Q_i)/M_i $, one can verify that both Theorem \ref{1 wasserstein theorem} and Theorem \ref{logarithmic wasserstein theorem} hold without the need to distinguish between the bounded and the unbounded cases. In particular, we point out that the crucial technical advantage here lies in the fact that if $x_0$ is picked according to the ad hoc probability just defined, it holds:

    \begin{align}\label{altra scelta possibile di P}
        &M_i \E [\widetilde{\W}_{1/2} (\delta_{\Phi (t,x_0)}, \delta_{\Phi_E (t,x_0)})] = M_i \int_{Q_i} \bar{\rho}_0 (x_0) \log \left( 1+ \frac{|\Phi_\delta (t,x_0) - \Phi_E (t,x_0)|}{\sqrt{h}}\right) dx_0 \\
        &=\int_{Q_i} \rho_0 (x_0) \log \left( 1+ \frac{|\Phi_\delta (t,x_0) - \Phi_E (t,x_0)|}{\sqrt{h}}\right) dx_0, \nonumber 
    \end{align} 

    while using a uniform distribution we needed to use that $ M_i / \Le ^d (Q_i) \leq \|\rho_0\|_{\infty}$, crucially relying on the boundedness of $\rho_0$. There are two reasons why we preferred to give the proofs using a uniform probability distribution even though it produces less sharp results. The first one is that a uniform probability distribution is simpler, and the second reason is that, in many practical cases, the initial distribution is bounded and the issue is given by the low regularity of the velocity field. We also notice that while using a bounded density we essentially only used $L^1$ stability estimates for the flow, if we had used the probability distribution just defined we would have really needed the full $L^p$ stability estimate on the flows. Indeed after summing the terms like \eqref{altra scelta possibile di P} for every $i$ we would have needed to use an estimate like:

    \begin{align*}   
    &\int_{\T^d} \rho_0 (x_0) \log \left( 1+ \frac{|\Phi_\delta (t,x_0) - \Phi_E (t,x_0)|}{\sqrt{h}}\right) dx_0 \leq \\
    &\|\rho_0 \|_{L^q (\T^d)} \left| \left| \log \left( 1+ \frac{|\Phi_\delta (t,\cdot) - \Phi_E (t,\cdot )|}{\sqrt{h}}\right) \right| \right|_{L^p(\T^d)}
    \end{align*}

    to conclude.
\end{remark}

Thanks to the above results we can now easily prove the following, which in a  certain sense is a Monte Carlo-like method. For the sake of simplicity and since the arguments would be identical, we give the result for a specific choice of $\Delta t, \Delta x$, $\alpha$ and only with respect to the $1$-Wasserstein distance.

\begin{proposition}\label{Monte-Carlo proposition}{\text{}}

    Let $u$, $\rho_0$ and $\{Q_i \}_{i=1} ^N$ satisfy the assumptions of Theorem \ref{1 wasserstein theorem}. Moreover, assume for simplicity $\alpha =1/2$, $\Delta t = \Delta x = h$ for some $h>0$ and that $C_t |\log(h)|^{-1} >> \sqrt{h}$. Then, defining 

    $$ S_n (t) \coloneqq \frac{1}{n} \sum_{i=1} ^n \rho_E ^i (t)$$

    the mean of $n$ realizations of $\rho_E$ (we omit to explicitly write for every $\rho_E ^i$ the dependence on $\bar{x}_0$) it holds that for every $k>1$

    $$ \bar{\P} \left(|\W_1 (S_n (t), \rho (t, \cdot)) - \E[\W_1 (S_n (t), \rho (t, \cdot))]| > k \sqrt{\frac{C_t |\log(h)|^{-1}}{n}} \right) \leq \frac{1}{k^2},$$

    where $C_t \to 0$ as $ t \to 0$ and $\E[\W_1 (S_n (t), \rho(t, \cdot))] \leq C_t |\log(h)|^{-1}$.
\end{proposition}
\begin{proof}

   Let $X_i \coloneqq \W_1 (\rho (t, \cdot), \rho_E ^i (t))$ be a random variable with respect to the probability $\bar{P}$ defined in \eqref{definition of the probability}. By Lemma \ref{splitting wasserstein}:

   \begin{equation*}
       \W_1 (S_n (t) , \rho(t, \cdot)) \leq \frac{1}{n} \sum_{i=1} ^n \W_1 (\rho_E ^i (t), \rho(t, \cdot))
   \end{equation*}

     so by \eqref{expected value desired estimate} 
     
    \begin{equation}\label{inequality for expected value of S_n}
    \E[\W_1 (S_n (t), \rho(t, \cdot))] \leq \frac{1}{n} \sum_{i=1} ^n \E[\W_1 (\rho_E ^i (t), \rho(t, \cdot))] \lesssim C_t |\log(h)|^{-1} ,
    \end{equation} 

    and we proved the upper bound on $\E[\W_1 (S_n (t), \rho(t, \cdot))]$. By construction the random variables $X_i$'s are independent and identically distributed, so by \eqref{variance inequality}

    \begin{equation}\label{variance of S_n estimate} \operatorname{Var} [\W_1 (S_n (t), \rho (t, \cdot))] = \frac{\operatorname{Var}[X_1]}{n} \leq \frac{C_t |\log(h)|^{-1}}{n}.
    \end{equation}

    Finally by Chebychev inequality and  \eqref{variance of S_n estimate}:

    \begin{align*}
    & \bar{\P} \left(|\W_1 (S_n (t), \rho (t, \cdot)) - \E[\W_1 (S_n (t), \rho (t, \cdot))]|> k \sqrt{\frac{C_t |\log(h)|^{-1}}{n}}\right) \\ &\leq \bar{\P} \left(|\W_1 (S_n (t), \rho (t, \cdot)) - \E[\W_1 (S_n (t), \rho (t, \cdot))]|> k \sqrt{\operatorname{Var}[\W_1 (S_n (t), \rho (t, \cdot))]}\right) \leq \frac{1}{k^2}.
    \end{align*}

\end{proof}

\section{Diffuse deterministic approximation of $\rho$}

In the previous section we showed how, in expected value, an approximation of the initial datum  by a sum of Dirac deltas advected by the vector field $\Phi_E$ led to an approximation of the density $u$ which solves \eqref{continuity equation}. The point was to prove that the analogue of the explicit Euler method applied to one of the simplest (weak) approximation of the initial datum (i.e. a sum of Dirac deltas) works also in this low-regularity setting. Such an approximation, however, is not well-suited for many applications where one really wants a diffuse approximation of the density rather than a singular one (for instance when there is a coupling between the continuity equation and other equations). Moreover, we want to also present an approximation scheme which works deterministically. We thus consider a variation of the previous results to cover also this important case, showing how allowing for a stronger approximation of the initial datum and a suitable approximation of the regular Lagrangian flow $\Phi$ we can give a deterministic approximation. As we have just pointed out there are two major problems with the method presented in Section 3:

\begin{enumerate}
    \item The approximation of $u$ is a singular measure. Often, one would like a diffuse approximation of a diffuse initial datum.

    \item The estimates are valid only in expected value, essentially due to the high sensitivity (at least a priori, see Remark \ref{remark that maybe estimates hold pointwise}) to the choice of the points $x_0^i \in Q_i$.

\end{enumerate}

A way to fix the above issues is to approximate the flow $\Phi$ by performing an additional averaging in space and to consider a strong (in $L^1$ norm) approximation of the initial datum. In this way, deterministic estimates for both the $1$-Wasserstein distance and for the logarithmic Wasserstein distance introduced in the previous section will follow easily, as we will see in Theorem \ref{teorema approssimazione diffusa}. Let us now make these ideas rigorous.

\begin{definition}\textbf{Definition of $\bar{\Phi}_E$}\label{definition of mean euler flow}

Assume the domain $\T^d$ is partitioned in sets $\{Q_i\}_{i=1} ^n$. Let $b$ satisfy conditions \eqref{main hypotheses} with $p>d$. For every $Q_i$ fix a point $x_0 ^i \in Q_i$. Given $ t \in (t_n , t_{n+1})$ and $x \in Q_i$, we define:

\begin{equation}\label{mean euler flow definition equation}
    \bar{\Phi}_E (t,x) \coloneqq \bar{\Phi}_E (t_n, x) + (t-t_n) \fint_{t_n} ^{t_{n+1}} \fint_{Q_i} b (s, \bar{\Phi}_{E} (t_n, x_0 ^i) +y-x_0 ^i) dy ds,
\end{equation}

where $0=t_0 < t_1 < \dots < t_M =T$ is a partition of $[0,T]$ with $t_{i+1} - t_i =\Delta t$ for some $\Delta t>0$ and $\Phi_E (0,x)= x \; \forall x \in \T^d$. In order to keep the notation as light as possible, we omit explicitly writing the dependence of $\bar{\Phi}_E$ on $\Delta t$.
\end{definition}

Before proving the stability estimate between $\Phi$ and $\bar{\Phi}_E$, we believe it is convenient to explicitly explain what the idea behind the definition of $\bar{\Phi}_E$ is and what approximation of $u$ it can yield. Since the argument does not depend on the particular element of the partition, we drop the subscripts. Consider a point $x \in Q$. In the time interval $[0, \Delta t]$, such a point is transported by $\bar{\Phi}_E$ to

$$ \bar{\Phi}_E (\Delta t,x)= x + \int_0 ^{\Delta t} \fint_{Q} b(s, y) dy ds= x + v. $$

This translation by $v$ is applied to every point in $Q$, so in the first time step we are just translating $Q$ by the vector $v=\int_0 ^{\Delta t} \fint_{Q} b(s, y) dy ds$. Let $\tilde{Q} \coloneqq Q + v$. In the second time step, we have defined $\bar{\Phi}_E$ so that it is constant on $\tilde{Q}$ and it acts on $\tilde{Q}$ by translating it by $\int_{\Delta t} ^{2\Delta t} \fint_{\tilde{Q}} b(s, y) dy ds$  . Indeed, consider $\tilde{x} \in \tilde{Q}$. There exists $x \in Q$ such that $\tilde{x} = \bar{\Phi}_E (\Delta t,x)=x+ v \in \tilde{Q}$. Then for every $\tilde{x} \in \tilde{Q}$: 

\begin{align*}
&\bar{\Phi}_E (2\Delta t, x)=  \bar{\Phi}_E (\Delta t, x) + \int_{\Delta t} ^{2\Delta t} \fint_{Q} b(s, \bar{\Phi}_E (\Delta t,x) + y - x ) dy ds\\
&= \tilde{x} +\int_{\Delta t} ^{2\Delta t} \fint_{Q} b(s,  y+ v  ) dy ds = \tilde{x} + \int_{\Delta t} ^{2 \Delta t} \fint_{\tilde{Q}} b(s,  z ) dz ds = \tilde{x} + w
\end{align*}

with $w$ not dependent on $\tilde{x}$. This is, in a sense, the same thing we have done when defining $\Phi_E$. The difference is that while with $\Phi_E$ we were only integrating $b$ in time (when $p>d)$ and we kept the space variable fixed, here we also average over some translation of $Q$. To sum up, if we partition the initial datum as:

$$ u_0 = \sum_{i=1} ^N u_0 \measurerestr Q_i = \sum_{i=1} ^{N} u_0 ^i ,$$

the proposed approximation of $u(t, \cdot)$ via $\bar{\Phi}_E$ will simply be:

$$ [\bar{\Phi}_E (t, \cdot)] _\# u_0 = \sum_{i=1} ^N [\bar{\Phi}_E (t, \cdot)] _\# u_0 ^i =\sum_{i=1} ^N u_0 ^i (x-\bar{\Phi}_E (t, x_0 ^i)) .$$

We now prove the stability estimate between $\Phi$ and $\bar{\Phi}_E$.

\begin{proposition}\label{estimate between Phi and  bar PhiE}{\text{}}

Consider  $\{Q_i\}_{i=1} ^{N}$ a partition of $\T^d$ such that $\operatorname{diam} (Q_i) \leq \Delta x \; \forall i$, for some $\Delta x>0$. Let $b$ be a vector field satisfying \eqref{main hypotheses} with $p>d$ and let $\bar{\Phi}_E$ be defined as in Definition \ref{definition of mean euler flow} with time step $\Delta t$. It holds that:

   $$\|\Phi (t, \cdot) - \bar{\Phi}_E (t, \cdot)\|_{L^p} \lesssim C_t |\log(\max\{\Delta t , \Delta x\})|^{-1} $$
    
    with $C_t \to 0$ as $t \to 0$.
    \end{proposition}

\begin{proof}
   Let $h \coloneqq \max\{\Delta t , \Delta x\}$. We mimic the proof of Proposition \ref{proposition Lp estimate regularised rlf and euler flow} with $\bar{\Phi}_E$ in place of $\Phi_E$. We recall that the following is a pointwise inequality, which holds for every $x, x_0 \in Q$ with $Q \in \{Q_i\}_{i=1} ^N$ and for this reason we write $\Phi(t), \bar{\Phi}_E (t)$ in place of $\Phi(t,x), \bar{\Phi}_E (t,x)$. We have that:

     \begin{align*} 
    &\log \left( 1 + \frac{|\Phi (t_n)- \bar{\Phi}_E (t_n)|}{h} \right) - \log \left( 1 + \frac{|\Phi (t_{n-1})- \bar{\Phi}_E (t_{n-1})|}{h} \right) \\
    &\leq \frac{h}{|\Phi (t_{n-1}) - \bar{\Phi}_E (t_{n-1})| + h } \frac{|\;|\Phi (t_n)- \bar{\Phi}_E (t_n)| - |\Phi (t_{n-1}) - \bar{\Phi}_E (t_{n-1})| \;|}{h}  \\
    &\leq \frac{|\Phi (t_n) - \Phi (t_{n-1}) -( \bar{\Phi}_E (t_n) - \bar{\Phi}_E (t_{n-1})) |}{|\Phi (t_{n-1}) - \bar{\Phi}_E (t_{n-1}) | + h} \\
    &= \frac{\left|\int_{t_{n-1}} ^{t_n} b (s, \Phi (s)) - \left(\fint_{Q} b(s, \bar{\Phi}_E (t_{n-1} , x_0) +y-x_0) dy \right) ds \right|}{|\Phi (t_{n-1}) - \bar{\Phi}_E (t_{n-1}) | + h}.
    \end{align*}

      For the sake of readability, we denote:
     
     $$ h_j(y) \coloneqq \bar{\Phi}_E (t_{j} , x_0) +y-x_0.$$

     By telescopic summing, as in \eqref{logarithmic estimate}, we end up having:
    \allowdisplaybreaks
    \begin{align}\label{non so come chiemarti, servi per un remark}
    &\log \left( 1 + \frac{|\Phi (t_n)- \bar{\Phi}_E (t_n))|}{h} \right) \leq \sum_{j=0} ^{n-1} \frac{\left|\int_{t_{j}} ^{t_{j+1}} b (s, \Phi (s)) - \left(\fint_{Q} b (s, h_j(y)) dy \right) ds \right|}{|\Phi (t_j) - \bar{\Phi}_E (t_j)|  + h}   \\
    &\leq   \sum_{j=0} ^{n-1} \frac{|\int_{t_{j}} ^{t_{j+1}} b (s, \Phi (s)) - b (s, \Phi (t_j)) ds| + \left|\int_{t_{j}} ^{t_{j+1}} b (s, \Phi(t_j)) - \left(\fint_Q b (s, h_j (y)) dy \right) ds\right|}{|\Phi (t_j) - \bar{\Phi}_E (t_j)|  + h}  \nonumber\\
    &\leq \underbrace{\sum_{j=0} ^{n-1} \frac{\int_{t_{j}} ^{t_{j+1}} |b (s, \Phi (s)) - b (s, \Phi (t_j))| ds }{|\Phi (t_j) - \bar{\Phi}_E (t_j)|  + h}}_{= \circled{1}} + \underbrace{\sum_{j=0} ^{n-1} \frac{\int_{t_{j}} ^{t_{j+1}} \left|b (s, \Phi(t_j)) - \left(\fint_Q b (s, h_j (y)) dy \right) \right| ds}{|\Phi (t_j) - \bar{\Phi}_E (t_j)|  + h}}_{= \circled{2}} .\nonumber 
    \end{align}

     The term $\circled{1}$ can be estimated exactly as in Proposition \ref{proposition Lp estimate regularised rlf and euler flow}. As for $\circled{2}$, we use that $\operatorname{diam}(Q) \leq \Delta x$, so arguing as in \eqref{estimate for p>d in the approximation}:

    \begin{align}\label{estimate of 2 in bar PhiE proposition}
    &\frac{\left|b(s, \Phi (t_j) ) - \fint_Q b(s, h_j(y)) dy  \right|}{|\Phi(t_j)- \bar{\Phi}_E (t_j) |+ h } \leq \frac{\fint_Q \left|b(s, \Phi (t_j) ) -  b(s, h_j(y))\right| dy  }{|\Phi(t_j)- \bar{\Phi}_E (t_j) |+ h } \\
    &\lesssim \frac{\fint_Q |\Phi(t_j) - h_j (y) |dy }{|\Phi(t_j)- \bar{\Phi}_E (t_j) |+ h } f(s,\Phi(t_j)) \leq f(s, \Phi(t_j )), \nonumber 
    \end{align}

    where we used the anisotropic estimate of Lemma \ref{lemma Caravenna Crippa}, $f$ is defined as in \eqref{definition of f for the estimate in p>d} and for the last inequality we used the fact that for every $y \in Q$:
    
    \begin{equation}\label{estimate of numerator of 2 in bar PhiE proposition} 
    |\Phi(t_j) - h_j (y) | \leq |\Phi(t_j ) - \bar{\Phi}_E (t_j)| + |y- x_0| \leq |\Phi(t_j ) - \bar{\Phi}_E (t_j)| + \Delta x .
    \end{equation}
   
    We conclude that

 $$ \sup_{t \in [0,T]} \left | \left | \log \left( 1 + \frac{|\bar{\Phi}_E (t,\cdot)- \Phi (t,\cdot)|}{h} \right) \right| \right|_{L^p} \leq C_t,$$

 with $C_t \to 0$ as $t\to 0$ arguing as in Proposition \ref{proposition Lp estimate regularised rlf and euler flow}, and given this bound on the $L^p$ norm we can conclude as in Proposition \ref{Lp stability estimate for regular Lagrangian flows}.
\end{proof} 

\begin{remark}\textbf{The restriction $p>d$}\label{remark p>d}

    The anisotropic estimate of Lemma \ref{lemma Caravenna Crippa} is crucial for our proof of Proposition \ref{estimate between Phi and  bar PhiE}, and that is the reason why we only considered the case $p>d$. If $1<p\leq d$ we do not have such an estimate, and we cannot estimate $\circled{2}$ using Lemma \ref{pointwise bv inequality}. Indeed, contrary to what we could do with $\Phi_E$, we cannot change variables when using $\bar{\Phi}_E$, since it is clearly not injective. A more refined argument may solve the issue for $1<p\leq d$, but this will need further investigations.
\end{remark}

We can now prove the following.

\begin{theorem}\label{teorema approssimazione diffusa}\textbf{Diffuse deterministic approximation of $u$}

    Consider a partition $\{Q_i\}_{i=1} ^N$ of $\T^d$, an initial density $u_0$ and a vector field $b$ such that $\operatorname{diam}(Q_i) \leq \Delta x$ for every  $i$ and $b$ satisfies \eqref{main hypotheses} with $p>d$. Then, given $\bar{\Phi}_E$ as in Definition \ref{definition of mean euler flow} with a time step $\Delta t$ and a function $\bar{u}_0$ such that $\|\bar{u}_0 - u_0\|_{L^1 (\T^d)} \leq \varepsilon $ for some $\varepsilon>0$, it holds that

    $$ \W_1 (u(t, \cdot) , [\bar{\Phi}_E (t, \cdot)] _\# \bar{u}_0) \lesssim \max\{C_t |\log(\max\{\Delta t , \Delta x\})|^{-1} , \varepsilon \}.$$

    Consider $\widetilde{\W}_1$ the Wasserstein distance defined in Definition \ref{wasserstein distance definition} with $\alpha =1$. It holds that

    $$ \widetilde{\W}_1 (u(t, \cdot) , (\bar{\Phi}_E (t, \cdot)) _\# \bar{u}_0) \lesssim \max\{C_t, |\log(\max\{\Delta t , \Delta x\})| \varepsilon \}.$$

\end{theorem}
\begin{proof}
  Let $h \coloneqq \max\{\Delta t , \Delta x\}$. By the triangle inequality we have

    \begin{align}\label{straightforward estimate wasserstein distance}
        \W_1 (u(t, \cdot) ,[\bar{\Phi}_E (t, \cdot)]_\# \bar{u}_0) \leq \W_1 (\overbrace{\Phi(t, \cdot) _\# u_0}^{=u(t, \cdot)}, \Phi(t, \cdot)_\# \bar{u}_0 ) + \W_1 (\Phi (t, \cdot)_\# \bar{u}_0, [\bar{\Phi}_E (t, \cdot)]_\# \bar{u}_0).
    \end{align}

    We can estimate the first term with $\|\bar{u}_0 - u_0\|_{L^1} \leq \varepsilon $ and the second one with

    $$ \int_{\T^d} \bar{u}_0 |\Phi- \bar{\Phi}_E| dx \leq \|\bar{u}_0 \|_{L^q (\T^d)} \|\Phi- \bar{\Phi}_E\|_{L^p (\T^d)} \lesssim C_t |\log(h)|^{-1}$$

    by Proposition \ref{estimate between Phi and  bar PhiE}. As for the estimate for $\widetilde{\W}_1$, it follows in the exact same way. We just notice that $1$-Lipschitz functions on $\T^d$ with respect to the distance $\d_{1}$, defined in \eqref{distance d_alpha} with $\alpha =1$, are uniformly bounded by $\approx |\log(h)|$ so we estimate the first addendum in \eqref{straightforward estimate wasserstein distance} (with $\widetilde{\W}_1$ in place of $\W_1$) with $|\log(h)|\varepsilon $, implicitly using Kantorovich-Rubinstein duality. In the same way we can estimate the second addendum with:

    \begin{align*} 
    \widetilde{\W}_1 (\Phi(t, \cdot) _\# \bar{u}_0, [\bar{\Phi}_E (t, \cdot)]_\# \bar{u}_0 )& \leq  \int_{\T^d} \bar{u}_0 (x) \log \left(1 + \frac{|\Phi(t,x) - \bar{\Phi}_E (t, x)|}{h} \right) dx \\
    &\leq \|\bar{u}_0 \|_{L^q (\T^d)} \left| \left| \log \left(1 + \frac{|\Phi(t,\cdot) - \bar{\Phi}_E (t, \cdot)|}{h} \right) \right| \right|_{L^p (\T^d)} \lesssim C_t
    \end{align*}

    by Proposition \ref{estimate between Phi and  bar PhiE}, and the proof is complete.
\end{proof}

We remark that this method is very easily implemented numerically and parallelisable. Moreover, it provides a better rate of convergence for the logarithmic Wasserstein distance than the one proved in \cite{schlichting2017convergence} (assuming $\Delta t = \Delta x =h/2$ for some $h>0$), having $h$ in the denominator in the definition of the logarithmic Wasserstein distance rather than $\sqrt{h}$. We finally remark that such an improvement in the case $p>d$ is in accordance with what we proved in Theorem \ref{logarithmic wasserstein theorem}.

\begin{remark}\label{rmk: no cfl} \textbf{The lack of a CFL condition}

A usual drawback of explicit methods is the so-called CFL condition, i.e. the requirement that $ \|b\|_\infty \Delta t / \Delta x \lesssim 1$. In this case, there is no such condition. In section 3 we worked essentially in a meshless setting, and the partition $\{Q_i\}_{i=1} ^{N}$ had the only role of providing an approximation of the initial datum by selecting a Dirac mass with suitable weight inside each $Q_i$. We had either a situation where both the time step $\Delta t$ and the space discretization scale $\Delta x$ contributed to the total error separately (Theorem \ref{1 wasserstein theorem}) or a situation where we gave estimates with respect to the maximum of the two (Theorem \ref{logarithmic wasserstein theorem}). Similar estimates also apply  to the diffuse case, where the error depends on $h = \max\{\Delta t, \Delta x\}$. We remark that the lack of a CFL condition may derive from the different point of view employed: while Eulerian methods (like the finite volume method) seem to be inherently affected by the need of a CFL condition, Lagrangian methods may have the advantage of not needing this additional requirement. We refer the reader to \cite{IDELSOHN2012168} for a broader discussion on such a matter (in the case of regular vector fields). 

\end{remark}

\begin{remark}[Weak vs. strong approximations of $\rho_0$]\label{weak approximation of rho0 remark}
In Section 3 we worked with a sum of Dirac deltas which approximated, in a weak sense, the initial datum $\rho_0$. Choosing such a measure was very natural, since we wanted a ``true'' approximation method, in the sense that it is truly implementable in a computer. Indeed in Section 3 we used $\Phi_E$ which is not locally constant in the space variable, so the only possibility was to use measures concentrated on points. In section 4, working with $\bar{\Phi}_E$ which is locally constant in space, we are more free to choose the initial datum, and actually we can compute the push-forward of every initial datum $\bar{\rho}_0$ as we explained  after Definition \ref{definition of mean euler flow}. In Theorem \ref{teorema approssimazione diffusa} we imposed that $\bar{\rho}_0$ approximated $\rho_0$ in a strong $L^1$ sense, so it is natural to wonder if the same result holds only assuming $\bar{\rho}_0$ to be a measure approximating $\rho_0$ in a weak sense (for instance, one could consider $\bar{\rho}_0$ to be a sum of Dirac deltas just like in Section 3). Given $\W$ the Wasserstein distance with respect to some distance $\d$, in order to estimate $\W(\rho(t, \cdot), [\bar{\Phi}_E (t, \cdot)]_\# \bar{\rho}_0) = \W(\Phi(t, \cdot)_\# \rho_0, [\bar{\Phi}_E (t, \cdot)]_\# \bar{\rho}_0)$ we crucially rely on the triangle inequality. Namely, we need to estimate either $\W(\Phi(t, \cdot)_\# \rho_0, \Phi(t, \cdot)_\#\bar{\rho}_0)$ or $\W([\bar{\Phi}_E (t, \cdot)]_\# \rho_0, [\bar{\Phi}_E (t, \cdot)]_\# \bar{\rho}_0)$. The problem is that, to the best of our knowledge, in general such distances can only be estimated with the $L^1$ norm of the difference between $\rho_0$ and $\bar{\rho}_0$. If $f$ is a Lipschitz function it actually holds that $\W(f_\# \mu, f_\# \nu) \leq \operatorname{Lip}(f) \W(\mu, \nu)$, however neither $\Phi$ nor $\bar{\Phi}_E$ are Lipschitz. And even if we previously approximated $u$ with some Lipschitz function $\tilde{u}$, the Lipschitz constant of the flow of $\tilde{u}$ would depend exponentially on both $\operatorname{Lip}(\tilde{u})$ and on the time $t$, which is clearly not ideal. 
\end{remark}

\section{Final discussion, questions and future perspectives}

We conclude by briefly summing up the results obtained, pointing out some questions and future possible research directions. In our paper, we proved a novel stability estimate between the regular Lagrangian flow of a vector field $u$ with Sobolev regularity in space and an approximation of such flow via a (sort of) explicit Euler method. We gave two approximations of the flow, i.e. $\Phi_E$ and $\bar{\Phi}_E$, for which the same stability estimate holds. We used that the solution $\rho$ of the continuity equation \eqref{continuity equation} is representable as a push-forward of the initial datum $\rho_0$ via the regular Lagrangian flow $\Phi$ of $u$ to approximate it using our approximations $\Phi_E$ and $\bar{\Phi}_E$ of $\Phi$ and different approximations of the initial datum $\rho_0$. In section 3 we used $\Phi_E$ to approximate $\Phi$ and a sum of Dirac deltas to approximate $\rho_0$. In this setting, we proved that the Dirac deltas, advected by $\Phi_E$, approximate $\rho$ in suitable Wasserstein distances. However such an approximation holds only in expected value, since a priori it strongly depends on the choice of the Dirac deltas.

\begin{question}
    Is the estimate truly valid in expected value only, or it can be made deterministic if $h= \Delta t =\Delta x$ is small enough?
    \end{question}

In section 4, we used a different approximation $\bar{\Phi}_E$ of $\Phi$ and a strong approximation $\bar{\rho}_0$ of $\rho_0$ in $L^1$ norm to prove a deterministic estimate. However, due to technical reasons in our proof of Proposition \ref{estimate between Phi and  bar PhiE}, we can prove that the crucial estimate on $\|\Phi(t, \cdot) - \bar{\Phi}_E (t, \cdot)\|_{L^p (\T^d)}$ holds only assuming $p>d$.

\begin{question}
    Does the stability estimate on $\|\Phi(t, \cdot) - \bar{\Phi}_E (t, \cdot)\|_{L^p (\T^d)}$ also hold for $1<p \leq d$? 
\end{question}

Moreover, as already said in Remark \ref{weak approximation of rho0 remark}, we do not know if a strong approximation of $\rho_0$ is truly needed in Theorem \ref{teorema approssimazione diffusa}:

\begin{question}
    Does Theorem \ref{teorema approssimazione diffusa} hold even assuming to have as an initial datum a measure $\bar{\rho}_0$ approximating $\rho_0$ in a weak sense (for example in some Wasserstein distance)?
\end{question}

Of course, it is totally possible that a better \textit{explicit} approximation of $\Phi$, that has nothing to do with either $\Phi_E$ or $\bar{\Phi}_E$ and for which there are no issues like the ones we just described above, exists. For example, in \cite{erroranalysistheta} we are considering an approximation of $\Phi$ based on the $\theta$-method, which however for $\theta \neq 0$ is \textit{implicit}. To conclude, we highlight possible future research directions. While numerical approximation methods for the continuity equation (see \cite{ben2019convergence,boyer2012analysis,jabin2024discretizing,schlichting2017convergence,schlichting2018analysis,walkington2005convergence}) and recently advection diffusion equation (\cite{navarro2023error}) have received some attention, more general situations where there is a non linear coupling between the density and the velocity field seem less studied. In this direction we cite \cite{jabin2024discretizing}, whose results can be applied also to some non linear systems. However, a limitation of the available methods is a certain ``rigidity" with respect to the mesh which seems somehow unavoidable (see \cite[Example 1.1]{jabin2024discretizing}) when using an \textit{explicit} finite volume scheme (we recall \cite{schlichting2018analysis}, where they do not need a Cartesian mesh, but using an \textit{implicit} upwind finite volume method). For this reason the mesh needs to be either Cartesian (as in \cite{schlichting2017convergence}) or at least satisfy some additional conditions (as in \cite{jabin2024discretizing}), which may be problematic in some situations. Our method does not need any assumption on the mesh other than an upper bound on their diameter and, eventually, a lower bound on their volume. Moreover it is easily parallelisable, but it only provides weak stability estimates in Wasserstein distance. 

\vspace{10pt} 

Future challenges could be the approximation of the solution of coupled systems without relying on any particular structure of the mesh, or developing methods based on the Lagrangian point of view of \eqref{continuity equation} that approximate the solution $\rho$ in a stronger topology, and for this reason we think that our point of view could be useful in such a task.

\section{Appendix}

We list here some results used in the above sections. We point out that the following results are stated in $\R^d$, while we have worked in $\T^d$. This is not a substantial problem since the estimates are of local nature.

\begin{lemma}{\textbf{(\cite{stein1970singular})}}\label{pointwise bv inequality}

Let $ f \in BV (\R^d)$. There exists a negligible set $N \subset \R^d$ and a constant $c_n >0$ such that

    $$ |f(x) - f(y)| \leq c_n |x-y| (M_\lambda Df(x) + M_\lambda Df (y)) \text{ for } x, y \in \R^d \setminus N \text{ with } |x-y| \leq \lambda.$$
\end{lemma}

\begin{lemma} {\textbf{(\cite[Lemma 5.1]{caravenna2021directional}})} \label{lemma Caravenna Crippa}

Let $f \in W^{1,p} (\R^d) $ with $p>d$. Then

    $$ |f(x) - f(y)| \leq c_{d,p} |x-y| [ M(|Df|^p) (x) ]^{1/p} \text{ for every } x,y \in \R^d .$$
\end{lemma}

\begin{lemma}{\textbf{(\cite[Theorem 2.1 and Proposition 2.3]{crippa2008estimates})}}\label{Lishitz rlf}

 Let $b$ be a bounded vector field belonging to $L^1 (0,T; W^{1,p} (\R^d)$ for some $p>1$ and with $div(b)^- \in L^1 (0,T;L^\infty (\R^d))$. Let $X: [0,T] \times \R^d \to \R^d$ be a regular Lagrangian flow for $b$. Let $L$ be the compressibility of the flow as defined in Definition \ref{rlf definition}. Define the quantity:

 $$ A_p (R,X)\coloneqq  \left[ \int_{B_R (0) }  \left( \sup_{0 \leq t \leq T} \sup_{0<r<2R} \fint_{B_r (x)} \log \left( 1+ \frac{|X(t,x)- X(t,y)|}{r} \right)dy\right)^p dx\right]^{1/p}.$$

 It holds that $A_p (R,X) \leq C(R, L, \|D_x b\|_{L^1 (L^p)})$. Moreover for every $\varepsilon>0$ and every $R>0$ we can find a set $K \subset B_R (0)$ such that $|B_R (0) \setminus K| \leq \varepsilon$ and for every $t \in [0,T]$ it holds

 $$ \operatorname{Lip} (X(t, \cdot) \measurerestr K) \leq \exp \left( \frac{c_d A_p (R, X)}{\varepsilon ^{1/p}}\right),$$

 where $c_d >0$ is a dimensional constant.
\end{lemma}

We finally report the proof of Lemma \ref{lemma soluzione cont eqn è push forward via rlf}:

\begin{proof}[Proof of Lemma \ref{lemma soluzione cont eqn è push forward via rlf}]
    We give for granted both the uniqueness of the solution of \eqref{continuity equation} and the existence of $\Phi$ (see \cite{diperna1989ordinary}, \cite{crippa2008estimates}), so we need only to verify that $\rho$ satisfies \eqref{continuity equation}, i.e.  

     $$ \int_0 ^T \int_{\T^d} \rho (\partial_t \xi + u\cdot \nabla_x \xi) dx dt + \int_{\T^d} \rho_0 (x) \xi (0, x) dx =0 \text{ for every } \xi \in C_c ^\infty ([0,T) \times \T^d).$$

      Indeed by definition, denoting with $\langle \cdot , \cdot \rangle$ the duality pairing:

     \begin{align*}
         \langle \partial_t \rho + \div_x (u\rho), \xi \rangle =  
 \int_0 ^T \int_{\T ^d}  \langle \partial_t \rho , \xi \rangle + \langle \div(\rho u) , \xi \rangle dx dt  .
     \end{align*}

   Now:

   $$ \int_0 ^T \langle \partial_t \rho , \xi \rangle dt  = [\rho \xi ]_{0} ^{T} -\int_0 ^T \rho \partial \xi dt = -\rho_0 (x) \xi (0, x) - \int_0 ^T \rho \partial_t \xi dt $$

   and 

   $$ \int_{\T^d} \div(u \rho) \xi dx = - \int_{\T^d} \rho u \cdot \nabla_x \xi dx.$$

   So $\rho$ solves \eqref{continuity equation} if

   $$  0= \langle \partial_t \rho + \div_x (u\rho), \xi \rangle = -\int_{\T^d} \rho_0 \xi (0, x) dx - \int_0 ^T \int_{\T^d} \rho ( \partial_t \xi + u \cdot \nabla_x \xi ) dx dt.$$
   
   Applying the push-forward formula in the second integral, since $\rho (t,\cdot) = \Phi (t,\cdot)_\# \rho_0 $:

    \begin{align*}
        &\int_0 ^T \int_{\T ^d} \rho [\partial_t \xi  + (\nabla_x \xi \cdot u)] dx dt\\
        &=  \int_0 ^T \int_{\T ^d}  \rho_0 [\partial_t \xi (t, \Phi(t,x)) + \nabla _x \xi (t, \Phi(t,x)) \cdot u(t, \Phi (t,x))] dx dt \\
        &= \int_0 ^T \int_{\T ^d} \rho_0 (x)\frac{d}{dt} \Big( \xi (t, \Phi(t,x)) \Big) dx dt= \int_{ \T^d}  \rho_0 (x) \left(\int_0 ^ T \frac{d}{dt}  \Big( \xi (t, \Phi(t,x))  \Big) dt\right) dx  \\
        & =- \int_{\T^d} \rho_0 (x) \xi (0,x) dx,
    \end{align*}
    and we conclude. 
 \end{proof}

\section{Acknowledgments}

The author is a member of the INdAM group GNAMPA. We wish to thank prof. Giovanni Alberti and prof. André Schlichting for the precious help and the suggestions offered in writing this article.

    \bibliographystyle{plain}
    \bibliography{bibliography.bib}

\end{document}